\documentclass[12pt]{amsart}

\newif\ifdebug
\debugfalse

\newif \iffig
\figfalse

\newif \iftable
\tabletrue

\usepackage{sfilip_package_settings}
\usepackage{sfilip_abbreviations}
\usepackage{sfilip_thm_style_long}
\usepackage{yuwei_abbreviations}

\usepackage{mathtools}

\title{Asymptotic shifting numbers in triangulated categories}

\hypersetup{
	pdftitle={Shifting numbers},	
	pdfauthor={Yu-Wei Fan, Simion Filip},	
	pdfsubject={math},		
	pdfkeywords={MSC},	
	pdfnewwindow=true,		
	linkcolor=darkblue,		
	citecolor=darkred,		
	filecolor=darkblue,		
	urlcolor=darkblue,		
	pdfborder={0 0 0},
	breaklinks=true
}

\thanks{Revised \textsc{\today} }

\date{August 2020}

\author{
	Yu-Wei Fan
}
\address{
	\parbox{0.7\textwidth}{
		Department of Mathematics\\
		University of California, Berkeley\\
		Berkeley, CA 94720, USA}
}
\email{{ywfan@berkeley.edu}}
\author{
	Simion Filip
}
\address{
	\parbox{0.7\textwidth}{
		Department of Mathematics\\
		University of Chicago\\
		5734 S University Ave\\
		Chicago, IL 60637}
}
\email{{sfilip@math.uchicago.edu}}

\begin{document}

\begin{abstract}
	We study invariants, called shifting numbers, that measure the asymptotic amount by which an autoequivalence of a triangulated category translates inside the category.
    The invariants are analogous to \Poincare translation numbers that are widely used in dynamical systems.
    We additionally establish that in some examples the shifting numbers provide a quasimorphism on the group of autoequivalences.
    Additionally, the shifting numbers are related to the entropy function introduced by Dimitrov, Haiden, Katzarkov, and Kontsevich, as well as the phase functions of Bridgeland stability conditions.
\end{abstract}

%
\maketitle
%
\tableofcontents


\section{Introduction}
	\label{sec:introduction}

This paper is concerned with categorical dynamical systems, namely endofunctors $F\colon \cD\to \cD$ on a triangulated category $\cD$.
We study invariants, the shifting numbers, that measure the asymptotic amount by which $F$ translates inside the triangulated category.


\subsection*{Translation and shifting numbers}
    \label{ssec:translation_and_shifting_numbers}

The concept is analogous to the \Poincare translation number, and our starting point is the central extension
\[
    0\to \bZ\to \Aut(\cD)\to \Aut(\cD)/[1]\to 1
\]
where $[1]$ denotes the shift functor of the ($\bZ$-graded) triangulated category $\cD$.

\paragraph{\Poincare translation and rotation numbers}
Let us recall some background on translation numbers that were introduced by Poincar\'e \cite{Poincare1885_Sur-les-courbes-definies-par-les-equations-differentielles-III} and have been used extensively in dynamics since then.
We refer to \cite{Ghys2001_Groups-acting-on-the-circle} for a general introduction to the theory.

Take $\bR/\bZ$ as the model of the circle and $\bR$ as its universal cover to obtain a central extension of groups
\[
0\ra\bZ\ra\Homeo_\bZ^+(\bR)\ra\Homeo^+(\bR/\bZ)\ra1
\]
where $\Homeo^+(\bR/\bZ)$ denotes the orientation-preserving homeomorphisms of the circle and $\Homeo_{\bZ}^+(\bR)$ denotes the orientation-preserving homeomorphisms of the real line commuting with translation by $\bZ$.
The \emph{Poincar\'e translation number} of an element $f\in\Homeo_\bZ^+(\bR)$ is defined by
\[
\wtilde\rho(f)\coloneqq\lim_{n\ra\infty}\frac{f^n(x_0)-x_0}{n}
\]
for some choice of basepoint $x_0\in \bR$ (it is a standard result that the limit always exists and is independent of the choice of $x_0$).
Here are some standard properties of the translation number:
\begin{align*}
    \wtilde\rho\left(f\circ T_k\right) & = \wtilde\rho\left(f\right)+k && \text{for the translation }T_k(x)=x+k\\
    \wtilde\rho\left(gfg^{-1}\right)& =\wtilde\rho\left(f\right) && \textbf{conjugacy invariance}\\
    \wtilde\rho\left(f^n\right) & =n\cdot \wtilde\rho\left(f\right) && \textbf{homogeneity}
\end{align*}
The \emph{rotation number} is defined for elements $f\in \Homeo^+(\bR/\bZ)$ by $\rho(f)\coloneqq\wtilde{\rho}(\wtilde{f}) \mod \bZ \in \bR/\bZ$ for any lift $\wtilde{f}$ of $f$.

\paragraph{Shifting numbers of autoequivalences}
In order to extend the above notions to the categorical setting, we need to first introduce some further concepts.
Instead of the basepoint $x_0\in \bR$ we will use a split generator $G\in \cD$, see \autoref{def:complexity_function}.
A number of categorical ``distance functions'' are available, see \autoref{ssec:the_definition}.
For definiteness, we will use the upper and lower $\Ext$-distance functions denoted by $\epsilon^+$ and $\epsilon^-$, see \autoref{def:extdistancefunction}, introduced in the study of Serre dimensions of triangulated categories in \cite{ElaginLunts,KikutaOuchiTakahashi}.
These are defined by
    \[
    \ep^+(E_1,E_2)\coloneqq\max\{k\in\bZ\colon\Hom(E_1,E_2[-k])\neq0\}
    \]
and
    \[
    \ep^-(E_1,E_2)\coloneqq\min\{k\in\bZ\colon\Hom(E_1,E_2[-k])\neq0\}.
    \]

\begin{theoremintro}[Shifting numbers and their properties]
    \label{thm:shifting_numbers_and_their_properties}
    Let $F\colon\cD\ra\cD$ be an endofunctor of a triangulated category $\cD$ and let $G$ be a split generator of $\cD$.
    The following limits exist and are finite real numbers:
    \[
    \tau^+(F)\coloneqq\lim_{n\ra\infty}\frac{\ep^+(G,F^nG)}{n}\text{ \ and \ }
    \tau^-(F)\coloneqq\lim_{n\ra\infty}\frac{\ep^-(G,F^nG)}{n},
    \]
    and furthermore are independent of the choice of split generator $G$.
    Each of $\tau^+$ and $\tau^-$ also satisfies the following properties:
    \begin{enumerate}
        \item For any $k\in\bZ$ we have
        \[
            \tau^\pm(F\circ[k])=\tau^\pm(F)+k.
        \]
        \item For any two endofunctors $F_1,F_2$ we have
        \[
            \tau^\pm(F_1F_2)=\tau^\pm(F_2F_1).
        \]
        In particular, if $F_2$ is an autoequivalence of $\cD$ then $\tau^\pm(F_2F_1 F_2^{-1})=\tau^\pm(F_1)$.
        \item For any $n\in \bN$ we have
        \[
            \tau^\pm(F^n)=n\cdot \tau^\pm(F)
        \]
        If $F$ is an autoequivalence and $\cD$ admits a Serre functor, then
        \[
            \tau^\pm(F^{-1})=-\tau^\mp(F).
        \]
        \item Setting $\tau(F)\coloneqq\frac{1}{2}\left(\tau^+(F)+\tau^-(F)\right)$.
        If $F$ is an autoequivalence and $\cD$ admits a Serre functor, then
        \[
            \tau(F^n)=n\cdot \tau(F) \text{ for any }n\in\bZ.
        \]
    \end{enumerate}
\end{theoremintro}

\noindent We will call $\tau^+(F)$ resp.~$\tau^-(F)$ the \emph{upper} resp.~\emph{lower} \emph{shifting numbers} of $F$, and $\tau(F)=\frac{1}{2}\left(\tau^+(F)+\tau^-(F)\right)$ the \emph{shifting number} of $F$.
For the proofs, see \autoref{sec:construction_and_properties_of_shifting_numbers}.

\paragraph{Bridgeland stability conditions}
An alternative definition of shifting numbers is possible using the notion of stability conditions, introduced by Bridgeland \cite{Bridgeland2007_Stability-conditions-on-triangulated-categories}.
It is also closer in spirit to the classical definition of the \Poincare translation number, but requires the existence of Bridgeland stability conditions.

\begin{theoremintro}[Shifting numbers from phases of stability conditions]
    \label{thm:shifting_numbers_from_the_phase_of_stability_conditions}
    Assume that $\cD$ admits a Serre functor and a Bridgeland stability condition $\sigma$.
    Let $\phi_\sigma^\pm\colon\Ob(\cD)\ra\bR$ be the phase functions with respect to $\sigma$ (see \autoref{def:stability}).

    Then the following limits exist, are independent of the choice of split generator $G$, and coincide with the upper/lower shifting numbers:
    \[
    \lim_{n\ra\infty}\frac{\phi^\pm_\sigma(F^nG)-\phi^\pm_\sigma(G)}{n}=\tau^\pm(F).
    \]
\end{theoremintro}

\noindent For the proof, see \autoref{theorem:computeviastability}.

\paragraph{Categorical entropy}
One can also connect the notion of shifting numbers with the categorical entropy function $h_t$ introduced by Dimitrov, Haiden, Katzarkov, and Kontsevich \cite{DimitrovHaidenKatzarkov_Dynamical-systems-and-categories}.

\begin{theoremintro}
\label{thm:shifting_entropy_intro}
Let $F\colon\cD\ra\cD$ be an endofunctor of a triangulated category $\cD$ with a split generator $G$, and let $h_t(F)$ be the categorical entropy function of $F$, see \autoref{def:catentropy}.

Then $h_t(F)$ is a real-valued convex function that satisfies:
\begin{align*}
    t\cdot \tau^+(F) & \leq h_t(F) \leq h_0(F) + t \cdot \tau^+(F) && \text{for }t\geq 0,\\
    t\cdot \tau^-(F) & \leq h_t(F) \leq h_0(F) + t \cdot \tau^-(F) && \text{for }t\leq 0.
\end{align*}
        In particular, we have $\lim_{t\ra\pm\infty}\frac{h_t(F)}{t}=\tau^\pm(F)\in\bR$.
\end{theoremintro}

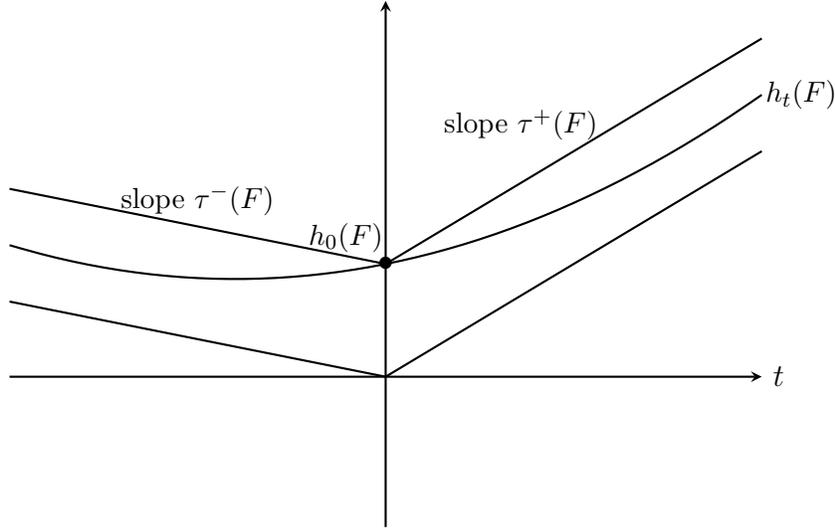
\begin{figure}[htbp]\centering
\begin{tikzpicture}[yscale=1,xscale=1,>=stealth]
\draw[thick,->] (-5,0) -- (5,0) node [right] {$t$};
\draw[thick,->] (0,-2) -- (0,5);
\draw[-,thick] (0,0) -- (5,3);
\draw[-,thick] (0,0) -- (-5,1);
\draw[-,thick] (0,1.5) -- (5,4.5) node [midway,above,black,xshift=-20] {\small slope $\tau^+(F)$};
\draw[-,thick] (0,1.5) -- (-5,2.5) node [midway,above,black] {\small slope $\tau^-(F)$};
\node[xshift=-15,yshift=10] at (0,1.5) {\small $h_0(F)$};
\node[] at (0,1.5) {$\bullet$};
\draw[domain=-5:5, smooth, variable=\x, thick] plot ({\x}, {1.5+0.2*\x+0.05*\x*\x});
\node[xshift=15] at (5,3.75) {\small $h_t(F)$};
\end{tikzpicture}
\caption{Bounds of the categorical entropy function $h_t(F)$}
\label{label}
\end{figure}
\noindent For the proofs, see \autoref{thm:convexity_and_finiteness} and \autoref{theorem:shiftingnumberviaentropy}.
We note that parts of the statements in Theorem~\ref{thm:shifting_numbers_and_their_properties} and \ref{thm:shifting_entropy_intro} also appeared in the work of Elagin and Lunts \cite[\S6]{ElaginLunts}, under the name of ``$F$-dimensions''.
We were led to these results independently of Elagin and Lunts and were made aware of their work by Genki Ouchi only after a first version of this text was posted on the arXiv.
The mass growth function $h_{\sigma,t}(F)$ also satisfies similar inequalities, see \autoref{theorem:computeviastability}.
In \autoref{ssec:LegendreTransform} we develop the observation that the shifting numbers determine the domain of definition of the Legendre-transformed entropy function.

\paragraph{Quasimorphisms}
An important property of the \Poincare translation number is that it gives a nontrivial \emph{quasimorphism} on $\Homeo_\bZ^+(\bR)$, namely there exists a constant $C>0$ such that
    \[
    |\wtilde\rho(fg)-\wtilde\rho(f)-\wtilde\rho(g)|\leq C
    \]
for any $f,g\in\Homeo_\bZ^+(\bR)$.
It is therefore natural to pose:
\begin{questionintro}
\label{ques:quasimorphism}
In what situations does the shifting number
    \[
    \tau\colon\Aut(\cD)\ra\bR
    \]
define a quasimorphism on the group of autoequivalences of a triangulated category $\cD$?
More generally, is there a quasimorphism $\phi$ on $\Aut(\cD)$ whose values satisfy $\phi(F)\in [\tau^-(F),\tau^+(F)]$?
\end{questionintro}

\noindent In the examples that we computed and are described below, occasionally we find that $\tau^+(F)=\tau^-(F)$ but this is special.
For instance, the spherical twist $T_S$ with respect to an $N$-spherical object $S$ has shifting numbers $\tau^+(T_S)=0>1-N=\tau^-(T_S)$ if $N\geq2$ (see \autoref{subsec:sphericaltwistsPtwists}).
In the case of the $A_2$ quiver and its $N$-Calabi--Yau category the average $\tau(F)$ of the upper and lower shifting numbers does give a quasimorphism.



\subsection*{Examples}
    \label{ssec:examples}

We compute several examples in this text, where we establish in particular that \autoref{ques:quasimorphism} has an affirmative answer.
See \autoref{theorem:standardautoequivalence}, \autoref{theorem:ellcurvequasihomom}, and \autoref{theorem:abelsurface} for the proof of the next result.

\begin{theoremintro}
Let $\cD=\cD^b\Coh(X)$ be the bounded derived category of coherent sheaves on a smooth projective variety $X$, where $X$
    \begin{itemize}
        \item is an elliptic curve, or
        \item is an abelian surface, or
        \item has ample or anti-ample canonical bundle $K_X$.
    \end{itemize}

Then $\tau(F)=\tau^\pm(F)$ for any $F\in\Aut(\cD)$, and
    \[
    \tau\colon\Aut(\cD)\ra\bR
    \]
is a quasimorphism. Moreover,
    \begin{itemize}
        \item when $X$ is an elliptic curve, $\tau$ can be decomposed into
                \[
                \Aut(\cD)\xrightarrow{s}\glstab\xrightarrow{t}\Homeo_\bZ^+(\bR)\xrightarrow{\wtilde{\rho}}\bR,
                \]
        where $s,t$ are homomorphisms, and $\wtilde\rho$ is the quasimorphism given by the Poincar\'e rotation number;
        \item when $X$ has ample or anti-ample canonical bundle, we have that $\Aut(\cD)=(\Aut(X)\ltimes\Pic(X))\times\bZ[1]$, and $\tau$ can be decomposed into
            \[
            \Aut(\cD)\xrightarrow{\pi}\bZ\xrightarrow{\iota}\bR,
            \]
        where $\pi$ is the projection to the $\bZ[1]$-factor, and $\iota$ is the natural inclusion of integers into real numbers.
    \end{itemize}
\end{theoremintro}

\noindent In general, computing the shifting number of an autoequivalence can be challenging, especially in situations arising from the composition of \emph{spherical twists} that do not commute.
We prove that \autoref{ques:quasimorphism} also has an affirmative answer for a group generated by two spherical twists with the simplest possible coupling, see \autoref{theorem:CYNA2quiver}.

\begin{theoremintro}
Let $N\geq3$ be an integer and let $\cD_N$ be the $N$-Calabi--Yau category associated to the $A_2$ quiver.
Consider the subgroup $\Aut_*(\cD_N)\subseteq\Aut(\cD_N)$ generated by the spherical twists $T_1,T_2$, and the shift $[1]$ (see \autoref{sec:CYA2}). Then
    \[
    \tau\colon\Aut_*(\cD_N)\ra\bR
    \]
is a quasimorphism.
More precisely, we have
    \[
    \tau=w-\frac16\big(\phi\circ\alpha\big)
    \]
where $w\colon\Aut_*(\cD_N)\ra\bQ$ and $\alpha\colon\Aut_*(\cD_N)\ra\PSL(2,\bZ)$ are group homomorphisms, and $\phi$ is the (homogenization of the) Rademacher function on $\PSL(2,\bZ)$.
\end{theoremintro}

\noindent Some recollections on the Rademacher function are provided in \autoref{app:quasihomomorphismPSL2}, see also \cite{BargeGhys1992_Cocycles-dEuler-et-de-Maslov}.

\paragraph{Quasimorphisms on central extensions of Lie groups}
It is a classical fact that Lie groups of hermitian type admit central $\bZ$-extensions (because the maximal compact has a nontrivial map to the circle).
The notion of \Poincare translation number generalizes to this setting, see for instance \cite{BurgerIozziWienhard2010_Surface-group-representations-with-maximal-Toledo-invariant}.
In \autoref{sec:quasimorphisms_on_lie_groups} we connect this construction with quasimorphisms on autoequivalences of abelian and K3 surfaces.
Recall that the Mukai lattice $\cN$ of such a surface has signature $(2,\rho)$ and the Lie group $\SO(\cN_{\bR})$ is thus of Hermitian type, and its universal cover has a quasimorphism.
\autoref{thm:quasimorphisms_for_abelian_surfaces} shows that in the case of abelian surfaces, the induced quasimorphism on $\Aut\cD$ coincides with the shifting number.



\subsection*{Further remarks}
    \label{ssec:further_remarks}

\paragraph{Categorical dynamics}
The study of \emph{categorical dynamical systems} $(\cD,F)$, i.e.~pairs consist of a triangulated category $\cD$ and an endofunctor $F\colon\cD\ra\cD$, was initiated in a paper of Dimitrov, Haiden, Katzarkov, and Kontsevich \cite{DimitrovHaidenKatzarkov_Dynamical-systems-and-categories}.
It has been an active area of research since then, see for instance \cite{DimitrovHaidenKatzarkov_Dynamical-systems-and-categories,
Fan2018_Entropy-of-an-autoequivalence-on-Calabi-Yau-manifolds,
2018arXiv180110485F,
Ikeda_Mass-growth-of-objects-and-categorical-entropy,
FFHKL,
FFO,Kikuta,KikutaOuchiTakahashi,KST20,KiTa,Ouchi2020_On-entropy-of-spherical-twists}.

The notion of categorical entropy $h_0(F)$, introduced in \cite{DimitrovHaidenKatzarkov_Dynamical-systems-and-categories}, captures the mass growth of objects under large iterates of $F$.
The shifting numbers $\tau^\pm(F)$ complement the entropy by measuring the phase growth of objects under large iterates of $F$.

\paragraph{Mirror symmetry and symplectic geometry}
All examples analyzed in this text are essentially algebro-geometric.
On the other hand, a wealth of quasimorphisms are available in symplectic geometry, see for instance Ruelle \cite{Ruelle1985_Rotation-numbers-for-diffeomorphisms-and-flows} for the origin of many constructions, and Entov--Polterovich \cite{EntovPolterovich2003_Calabi-quasimorphism-and-quantum-homology} for more recent constructions based on Floer theory.
It would be interesting to investigate what kind of quasimorphisms one obtains by the methods of the present text in the case of Fukaya categories.
It would be additionally interesting to understand if the construction of Bestvina--Fujiwara \cite{BestvinaFujiwara2002_Bounded-cohomology-of-subgroups-of-mapping-class-groups} of quasimorphisms can be extended to the symplectic case, or to the case of K3 surfaces using Bridgeland's conjecture \cite{Bridgeland2008_Stability-conditions-on-K3-surfaces}.
The reader can find many further stimulating questions in Smith's survey \cite{Smith2018_Stability-conditions-in-symplectic-topology}.

\paragraph{Some applications of quasimorphisms}
Let us note that the existence of nontrivial quasimorphisms on a group has purely algebraic consequences for the structure of the group, see for instance
\cite{Kotschick2004_Quasi-homomorphisms-and-stable-lengths-in-mapping-class-groups}.
For instance, when a group is perfect it carries a stable commutator length function, which is nontrivial if and only if there exists a nontrivial quasimorphism by Bavard's theorem \cite{Bavard}.
For more on stable commutator length, see Calegari's monograph \cite{Calegari2009_scl}.
It is therefore interesting to ask also the following:

\begin{questionintro}[Perfectness of autoequivalence groups]
    \label{eg:perfectness_of_autoequivalence_groups}
    For which triangulated categories $\cD$ is the group $\Aut(\cD)/[1]$ a perfect group?
    Recall that a perfect group is one where every element is a finite product of commutators.
\end{questionintro}

\noindent Let us also note that it would be interesting to understand the relation between the constructions in this paper and quasimorphisms on the universal cover of the group of contact diffeomorphisms of a contact manifold constructed by Givental \cite{Givental1990_The-nonlinear-Maslov-index} and Eliashberg--Polterovich \cite{EliashbergPolterovich2000_Partially-ordered-groups-and-geometry-of-contact-transformations}.
We are grateful to Leonid Polterovich for bringing these results to our attention.

\paragraph{Analogies}
We end with a comparison between the motivating concepts in dynamical systems and their categorical counterparts.

\newcommand{\mc}[3]{\multicolumn{#1}{#2}{#3}}

\renewcommand{\arraystretch}{1.1}

\vskip 1em

    \begin{tabular}{p{0.4\textwidth} p{0.5\textwidth}}
        \toprule
        \multicolumn{1}{c}{\textbf{Translation numbers} } &  \multicolumn{1}{c}{ \textbf{Shifting numbers} }\\
        \midrule
        $f\in\Homeo_\bZ^+(\bR)$ & $F\colon\cD\ra\cD$ \\
        \midrule
        basepoint $x_0\in\bR$ & split generator $G\in\cD$\\
        \midrule
        amount of translation & phase of stability condition\\
        \midrule
        $f^n(x_0)-x_0$ & $\phi_\sigma^\pm(F^nG)-\phi_\sigma^\pm(G)$ \\
        \midrule
        translation number $\wtilde{\rho}$ & upper/lower shifting numbers $\tau^\pm$\\
        \bottomrule
     \end{tabular}
\vskip 1em

\paragraph{Conventions}
Let $\tbk$ be a base field.
Throughout this article, all triangulated categories are assumed to be $\tbk$-linear, $\bZ$-graded, saturated, and of finite type (i.e.~the $\tbk$-vector space $\oplus_i\Hom_\cD(E,F[i])$ is finite-dimensional for any pair of objects $E,F$ in the category).
Functors between triangulated categories are assumed to be $\tbk$-linear, triangulated, and not virtually zero (i.e.~any power is not the zero functor).

\paragraph{Acknowledgments}
We are grateful to Leonid Polterovich and Ivan Smith for insightful comments on a preliminary version of the manuscript.
We are grateful to Genki Ouchi for pointing out the reference \cite{ElaginLunts} after the first version of the present article was posted on the arXiv.

YF would like to thank Emanuele~Macr\`i for helpful discussions during the early stages of the project.
This research was partially conducted during the period SF served as a Clay Research Fellow.
SF gratefully acknowledges support from the Institute for Advanced Study.
This material is based upon work supported by the National Science Foundation under Grant No. DMS-2005470.
This material is based upon work supported by the National Science Foundation under Grant No. DMS-1638352, DMS-1107452, 1107263, 1107367 ``RNMS: Geometric Structures and Representation Varieties'' (the GEAR Network).
This material is based upon work supported by the National Science Foundation under Grant No. DMS-1440140 while SF was in residence at the Mathematical Sciences Research Institute in Berkeley, California, during the Fall 2019 semester.




\section{Construction and properties of shifting numbers}
	\label{sec:construction_and_properties_of_shifting_numbers}

\paragraph{Outline of section}
Inspired by the classical construction of the translation and rotation numbers, we study the notion of shifting number of an endofunctor of a triangulated category and some of its basic properties.


\subsection{The definition}
	\label{ssec:the_definition}

The translation number of an element $f\in\Homeo^+_\bZ(\bR)$ measures the average displacement of points in an orbit $\{f^n(x_0)\}$.
To study its categorical analogue, one needs a notion of ``distance'' between pairs of objects in a triangulated category.
We first recall the \emph{complexity function} introduced in \cite{DimitrovHaidenKatzarkov_Dynamical-systems-and-categories}.

\begin{definition}[Complexity function {\cite[Definition~2.1]{DimitrovHaidenKatzarkov_Dynamical-systems-and-categories}}]
    \label{def:complexity_function}
Let $E_1,E_2$ be objects in a triangulated category $\cD$.
The \emph{complexity function} of $E_2$ relative to $E_1$ is the function $\delta_t(E_1,E_2)\colon\bR\to[0,\infty]$ given by
    \[
    \delta_t(E_1,E_2)\coloneqq
    \inf
    \left\{\sum_{k=1}^me^{n_kt}\Big|
                {\substack{
                0=A_0\to A_1\to\cdots\to A_m= E_2\oplus F\text{ for some }F\in\Ob(\cD), \\
                \text{where }\Cone(A_{k-1}\to A_k)\cong E_1[n_k]\text{ for all }k
                }}\right\}
    .
    \]
Define $\delta_t(E_1,E_2)=0$ if $E_2\cong0$, and define $\delta_t(E_1,E_2)=\infty$ if $E_2$ does not lie in the thick triangulated subcategory generated by $E_1$.

An object $G$ is called a \emph{split generator} if for any $E\in \cD$ we have that $\delta_t(G,E)<+\infty$.
\end{definition}
The notion of complexity function was introduced in \cite{DimitrovHaidenKatzarkov_Dynamical-systems-and-categories} in order to define the \emph{categorical entropy function} of an endofunctor.
Another important distance function, called the Ext-distance function in \cite{FFO}, is obtained by computing the dimensions of morphism spaces.

\begin{definition}[Ext-distance function and upper/lower Ext-distance]
\label{def:extdistancefunction}
Let $E_1,E_2$ be objects in a triangulated category $\cD$.
    \begin{itemize}
        \item The \emph{Ext-distance function} from $E_1$ to $E_2$ is the function of the $t$-variable $\ep_t(E_1,E_2)\colon\bR\to\bR_{\geq0}$ given by
        \[
        \ep_t(E_1,E_2)\coloneqq\sum_{k\in\bZ}\dim_\tbk\Hom(E_1,E_2[-k])e^{kt}.
        \]
        \item The \emph{upper} and \emph{lower $\Ext$-distances} from $E_1$ to $E_2$ are defined to be
        \begin{align*}
        \ep^+(E_1,E_2)&\coloneqq\max\{k\in\bZ\colon\Hom(E_1,E_2[-k])\neq0\},
        	\intertext{ and }
        \ep^-(E_1,E_2)&\coloneqq\min\{k\in\bZ\colon\Hom(E_1,E_2[-k])\neq0\}.
        \end{align*}
    \end{itemize}
\end{definition}

\begin{remark}
The upper and lower $\Ext$-distances $\ep^\pm(E_1,E_2)$ were used in the study of Serre dimensions of triangulated categories \cite{ElaginLunts,KikutaOuchiTakahashi}.
They are not defined if $\Hom^\bullet(E_1,E_2)=0$.
We will only use them for $E_1=G$ and $E_2=F^nG$, where $G$ is a split generator of $\cD$ and $F$ is an endofunctor of $\cD$ (which is not virtually zero).
By the proof of \cite[Theorem~2.7]{DimitrovHaidenKatzarkov_Dynamical-systems-and-categories} we have that $\Hom^\bullet(G,F^nG)\neq0$ and so $\ep^\pm(G,F^nG)$ is well-defined.
\end{remark}

\begin{definition}[Categorical entropy function {\cite[Definition~2.5]{DimitrovHaidenKatzarkov_Dynamical-systems-and-categories}}]
\label{def:catentropy}
Let $F\colon\cD\to\cD$ be an endofunctor of a triangulated category $\cD$ with a split generator $G$.
The \emph{categorical entropy function} of $F$ is the function $h_t(F)\colon\bR\to[-\infty,\infty)$ in variable $t$ given by
    \[
    h_t(F)\coloneqq\lim_{n\to\infty}\frac{\log\delta_t(G,F^nG)}{n}.
    \]
\end{definition}
\noindent The following result summarizes \cite[Lemma~2.6, Theorem~2.7]{DimitrovHaidenKatzarkov_Dynamical-systems-and-categories}.

\begin{theorem}
    \label{thm:entropy_via_ext_distance}
Let $F\colon\cD\to\cD$ be an endofunctor of a triangulated category $\cD$ with a split generator $G$.
Then the limit in \autoref{def:catentropy} defining the categorical entropy function $h_t(F)$ is independent of the choice of generator.
Moreover, it can be computed alternatively via the Ext-distance function:
    \[
    h_t(F)=\lim_{n\to\infty}\frac{\log\ep_t(G,F^nG)}{n}.
    \]
\end{theorem}

Before introducing the definition of shifting numbers, we first prove that the value $-\infty$ can be excluded for the categorical entropy function.

\begin{theorem}[Convexity and finiteness of categorical entropy function]
    \label{thm:convexity_and_finiteness}
    The categorical entropy function $h_t(F)$ is a real-valued convex function in $t$.
\end{theorem}
\begin{proof}
    Consider the functions
    \[
        h_{t,n}(F)\coloneqq \frac{\log \epsilon_t\left(G,F^{n}G\right)}{n} \text{ appearing in \autoref{thm:entropy_via_ext_distance}.}
    \]
    Note that each function $h_{t,n}(F)$ is convex in $t$, as can be verified by differentiating twice any function of the form $\log\left(\sum a_ie^{k_i\cdot t}\right)$ with $a_i\geq 0$ and applying Cauchy--Schwarz.
    Since $h_{t}(F)$ is a pointwise limit of convex functions, it is itself convex.

    Next, we do know that $h_{0}(F)\geq 0$ from \autoref{def:catentropy}, since the defining functions always satisfy that.
    Additionally, we also know that $h_t(F)<+\infty$ for all $t$ since the existence of the limit established in \cite{DimitrovHaidenKatzarkov_Dynamical-systems-and-categories} is via Fekete's lemma, in particular it equals the infimum of the sequence $\frac{1}{n}\log \delta_t(G,F^nG)$, and for $n=1$ this is already finite.
    Together with convexity, this excludes the possibility that $h_t(F)=-\infty$ for some $t\neq 0$, since convexity would force $h_{-t}(F)=+\infty$ which is a contradiction.
\end{proof}


\begin{theorem}[Shifting numbers via entropy, see also \cite{ElaginLunts}]
\label{theorem:shiftingnumberviaentropy}
Let $F\colon\cD\to\cD$ be an endofunctor of a triangulated category $\cD$ with a split generator $G$.
Then the following limits exist (in $\bR$)
    \[
    \tau^+(F)\coloneqq\lim_{n\to\infty}\frac{\ep^+(G,F^nG)}{n}\text{ \ and \ } \tau^-(F)\coloneqq\lim_{n\to\infty}\frac{\ep^-(G,F^nG)}{n},
    \]
and are independent of the choice of split generator $G$.
Moreover, the limits $\lim_{t\to\pm\infty}\frac{h_t(F)}{t}$ also exist, and we have
    \[
    \lim_{t\to\pm\infty}\frac{h_t(F)}{t}=\tau^\pm(F).
    \]
Additionally, the following inequalities for the entropy function hold:
\begin{align*}
    t\cdot \tau^+(F) & \leq h_t(F) \leq h_0(F) + t \cdot \tau^+(F) && \text{for }t\geq 0,\\
    t\cdot \tau^-(F) & \leq h_t(F) \leq h_0(F) + t \cdot \tau^-(F) && \text{for }t\leq 0.
\end{align*}

\end{theorem}

\begin{definition}[Upper and lower shifting numbers]
	\label{def:upper_and_lower_shifting_numbers}
	For an endofunctor $F\colon \cD\to \cD$, the quantities $\tau^+(F)$ and $\tau^-(F)$ are called the \emph{upper}, resp. \emph{lower}, \emph{shifting numbers} of $F$.
\end{definition}

\begin{proof}[Proof of \autoref{theorem:shiftingnumberviaentropy}]
The proof is essentially given in \cite[Proposition~6.13]{ElaginLunts}, except for the precise upper bound on $h_t(F)$.
We are grateful to Genki Ouchi, who indicated to us the paper of Elagin and Lunts after a first version of this text appeared on the arXiv.

Suppose first that $t>0$.
Then we have
    \[
    e^{\ep^+(G,F^nG)t}\leq
    \ep_t(G,F^nG)\leq
    \Big(\sum_k\dim_\tbk\Hom(G,F^nG[-k])\Big)\cdot e^{\ep^+(G,F^nG)t}.
    \]
By applying $\log(-)/n$ and taking $n\to\infty$, one obtains
    \begin{align*}
        t\cdot\limsup_{n\to\infty}\frac{\ep^+(G,F^nG)}{n} & \leq
        \limsup_{n\to\infty}\frac{\log\ep_t(G,F^nG)}{n} \\
        & =\lim_{n\to\infty}\frac{\log\ep_t(G,F^nG)}{n} =h_t(F) \\
        & =\liminf_{n\to\infty}\frac{\log\ep_t(G,F^nG)}{n} \\
        & \leq\liminf_{n\to\infty}\frac{\log\Big(\sum_k\dim_\tbk\Hom(G,F^nG[-k])\Big)+\ep^+(G,F^nG)t}{n} \\
        & =h_0(F)+t\cdot\liminf_{n\to\infty}\frac{\ep^+(G,F^nG)}{n}.
    \end{align*}
Hence
\begin{align}
    \label{eqn:tau_t_prelimit}
    t\cdot \limsup_{n\to\infty}\frac{\ep^+(G,F^nG)}{n}
    \leq {h_t(F)}
    \leq
    {h_0(F)}+t\cdot \liminf_{n\to\infty}\frac{\ep^+(G,F^nG)}{n}.
\end{align}
Note that $h_0(F)$ is a non-negative real number (we always assume that the functors are not virtually zero).
Dividing by $t$ and sending it to $+\infty$, it follows that the two limits exist and coincide:
    \[
    \lim_{t\to+\infty}\frac{h_t(F)}{t}=\lim_{n\to\infty}\frac{\ep^+(G,F^nG)}{n}
    \]
Knowing that the limits exist and equal to $\tau^+(F)$ and returning to \autoref{eqn:tau_t_prelimit}, the inequality
\[
    t\cdot \tau^+(F) \leq h_t(F) \leq h_0 + t\cdot \tau^+(F) \quad\text{ follows for }t\geq 0.
\]
Similarly to the above reasoning, for $t<0$ we have
    \[
    e^{\ep^-(G,F^nG)t}\leq
    \ep_t(G,F^nG)\leq
    \Big(\sum_k\dim_\tbk\Hom(G,F^nG[k])\Big)\cdot e^{\ep^-(G,F^nG)t}.
    \]
Hence
    \[
     t\cdot\liminf_{n\to\infty}\frac{\ep^-(G,F^nG)}{n}\leq h_t(F)\leq h_0(F)+t\cdot\limsup_{n\to\infty}\frac{\ep^-(G,F^nG)}{n},
    \]
and therefore
    \[
    \liminf_{n\to\infty}\frac{\ep^-(G,F^nG)}{n}\geq\frac{h_t(F)}{t}\geq\frac{h_0(F)}{t}+\limsup_{n\to\infty}\frac{\ep^-(G,F^nG)}{n}.
    \]
By letting $t\to-\infty$, we prove the following two limits both exist and coincide:
    \[
    \lim_{t\to-\infty}\frac{h_t(F)}{t}=\lim_{n\to\infty}\frac{\ep^-(G,F^nG)}{n}
    \]
and the claimed inequality for $t\leq 0$ and $h_t(F)$ follows as before.

Finally, we note that $\tau^\pm(F)\in\bR$ follows from the above inequalities and the finiteness of $h_t(F)$ (\autoref{thm:convexity_and_finiteness}).
\end{proof}

\begin{remark}\leavevmode
	\label{remark:rigidity_and_disagreement_of_shifting_numbers}
	\begin{enumerate}
		\item If $h_0(F)=0$ then from \autoref{theorem:shiftingnumberviaentropy} it follows that there exist constants $\tau^{\pm}(F)\in \bR$ such that $h_t(F)=\tau^{\pm}(F)\cdot t$ for $t\geq 0$, resp. $t\leq 0$.
		\item In general, the upper and lower shifting numbers of an endofunctor do not have to coincide, cf. examples of spherical twists in \autoref{subsec:sphericaltwistsPtwists}.
	\end{enumerate}
\end{remark}

In \autoref{sec:examples_of_shifting_numbers}, we will see that in some situations the upper and lower shifting numbers agree, and give a quasimorphism, and sometimes they don't agree.
When they don't agree, such as the example of the $A_2$ quiver in \autoref{sec:CYA2}, it is useful to introduce the average of the two quantities:

\begin{definition}
\label{def:shiftingnumber}
Let $F\colon\cD\to\cD$ be an endofunctor of a triangulated category $\cD$ with a split generator $G$.
Define the \emph{shifting number of $F$} to be the average
    \[
    \tau(F)\coloneqq\frac{\tau^+(F)+\tau^-(F)}{2}.
    \]
\end{definition}

\subsection{Shifting numbers via stability conditions}
	\label{ssec:shifting_numbers_via_stability_conditions}

\subsubsection{Setup}
	\label{sssec:setup_shifting_numbers_via_stability_conditions}

In this section, we show that if a triangulated category $\cD$ admits a Bridgeland stability condition $\sigma$, then the upper/lower shifting numbers of $F$ coincide with the average displacements of $\{\phi_\sigma^\pm(F^nG)\}\subseteq\bR$, where $\phi_\sigma^\pm\colon\Ob(\cD)\ra\bR$ are the phase functions with respect to $\sigma$.
We first recall the notion of \emph{Bridgeland stability conditions} on triangulated categories.

\begin{definition}[Bridgeland \cite{Bridgeland2007_Stability-conditions-on-triangulated-categories}]
\label{def:stability}
Let $\cD$ be a triangulated category and let $\cl\colon K_0(\cD)\ra\Gamma$ be a group homomorphism from the Grothendieck group of $\cD$ to a finite rank free abelian group $\Gamma$.
A \emph{Bridgeland stability condition} $\sigma=(Z_\sigma,P_\sigma)$ on $\cD$ consists of a group homomorphism $Z_\sigma\colon\Gamma\ra\bC$ (\emph{central charge}), and a collection of full additive subcategories $P_\sigma=\{P_\sigma(\phi)\}_{\phi\in\bR}$ of $\cD$ (\emph{$\sigma$-semistable objects of phase $\phi$}), such that:
\begin{enumerate}
    \item $Z_\sigma(E)\coloneqq Z_\sigma(\cl([E]))\in\bR_{>0}\cdot e^{i\pi\phi}$ for any $0\neq E\in P_\sigma(\phi)$,
    \item $P_\sigma(\phi+1)=P_\sigma(\phi)[1]$ for any $\phi\in\bR$,
    \item\label{def:noHombigtosmall} $\Hom(E_1,E_2)=0$ if $E_i\in P_\sigma(\phi_i)$ and $\phi_1>\phi_2$,
    \item for any $E\in\cD$,  there exist exact triangles $E_{i-1}\ra E_i\ra A_i\xrightarrow{+1}$ for $1\leq i\leq n$, such that $E_0=0$, $E_n=E$, $A_i\in P_\sigma(\phi_i)$, and $\phi_1>\cdots>\phi_n$,
    \item\label{def:supportproperty} (support property \cite{KontsevichSoibelman08}) there exist a constant $C>0$ and a norm $\norm{\cdot}$ on $\Gamma\otimes_\bZ\bR$ such that $\norm{\cl([E])}\leq C|Z_\sigma(E)|$ for any $0\neq E\in\cup_{\phi\in\bR}P_\sigma(\phi)$.
\end{enumerate}
\end{definition}

The collection of exact triangles in (iv) is called the \emph{Harder--Narasimhan filtration} of $E$ and the objects $A_i$ are called the $\sigma$-semistable factors.
The maximal and minimal phases in the filtration are denoted by $\phi_\sigma^+(E)\coloneqq\phi_1$ and $\phi_\sigma^-(E)\coloneqq\phi_n$ and define real-valued functions:
    \[
    \phi_\sigma^\pm\colon\Ob(\cD)\ra\bR.
    \]
Note that when $E$ is a $\sigma$-semistable object, $\phi_\sigma(E)=\phi^+_\sigma(E)=\phi^-_\sigma(E)$ is nothing but the rotation angle of its central charge $Z_\sigma(E)\in\bC$.

Next we recall the definitions of mass function and mass growth function with respect to a Bridgeland stability condition.

\begin{definition}[Mass functions with respect to stability conditions {\cite[\S4.5]{DimitrovHaidenKatzarkov_Dynamical-systems-and-categories}},\cite{Ikeda_Mass-growth-of-objects-and-categorical-entropy}]
Let $E$ be a non-zero object in a triangulated category $\cD$ and let $\sigma$ be a Bridgeland stability condition on $\cD$.
The \emph{mass function} of $E$ with respect to $\sigma$ is the function of the $t$-variable $m_{\sigma,t}(E)\colon\bR\ra\bR_{>0}$ given by
    \[
    m_{\sigma,t}(E)\coloneqq\sum_{k=1}^m|Z_\sigma(A_k)|e^{\phi_\sigma(A_k)t},
    \]
where $A_1,\ldots,A_m$ are the $\sigma$-semistable factors of $E$.
\end{definition}

\begin{definition}[Mass growth function {\cite[\S4.5]{DimitrovHaidenKatzarkov_Dynamical-systems-and-categories}},\cite{Ikeda_Mass-growth-of-objects-and-categorical-entropy}]
Let $F\colon\cD\ra\cD$ be an endofunctor of a triangulated category $\cD$. The \emph{mass growth function} of $F$ with respect to a Bridgeland stability condition $\sigma$ on $\cD$ is the function $h_{\sigma,t}(F)\colon\bR\ra[-\infty,\infty)$ in variable $t$ given by
    \[
    h_{\sigma,t}(F)\coloneqq\sup_{E\in\cD}\Bigg\{\limsup_{n\ra\infty}\frac{\log m_{\sigma,t}(F^nE)}{n}\Bigg\}.
    \]
\end{definition}

\noindent The relationship between $h_t(F)$ and $h_{\sigma,t}(F)$ was suggested in \cite[\S4.5]{DimitrovHaidenKatzarkov_Dynamical-systems-and-categories} and later proved by Ikeda \cite{Ikeda_Mass-growth-of-objects-and-categorical-entropy}.

\begin{theorem}[{\cite[Theorem~1.1]{Ikeda_Mass-growth-of-objects-and-categorical-entropy}}]
\label{theorem:Ikedamassentropybound}
With notation as above, assume that $\cD$ has a split generator $G$.
Then
    \[
    h_{\sigma,t}(F)=\limsup_{n\ra\infty}\frac{\log m_{\sigma,t}(F^nG)}{n}.
    \]
Moreover,
    \[
    h_{\sigma,t}(F)\leq h_t(F)<\infty
    \]
for any stability condition $\sigma$, and $h_{\sigma,t}(F)=h_{\sigma',t}(F)$ if $\sigma$ and $\sigma'$ lie in the same connected component of the space of Bridgeland stability conditions on $\cD$.
\end{theorem}

\noindent In general, it is not known whether the categorical entropy $h_t(F)$ always coincides with the mass growth $h_{\sigma,t}(F)$.
Nevertheless, we prove in the following theorem that their linear growth rates at infinity coincide for \emph{any} Bridgeland stability condition $\sigma$. 
Moreover, the average displacements of $\{\phi_\sigma^\pm(F^nG)\}\subseteq\bR$ also give the same numbers.

\begin{theorem}[Shifting numbers via stability conditions]
	\label{theorem:computeviastability}
Let $F\colon\cD\ra\cD$ be an endofunctor of a triangulated category $\cD$ with a split generator $G$, and let $\sigma$ be a Bridgeland stability condition on $\cD$. Then the following limits exist and coincide with the upper shifting number:
    \[
    \lim_{n\ra\infty}\frac{\phi^+_\sigma(F^nG)}{n}=\lim_{t\ra\infty}\frac{h_{\sigma,t}(F)}{t}=\tau^+(F).
    \]
Moreover, if $\cD$ admits a Serre functor, then the following limits exist and coincide with the lower shifting number:
    \[
    \lim_{n\ra\infty}\frac{\phi^-_\sigma(F^nG)}{n}=\lim_{t\ra-\infty}\frac{h_{\sigma,t}(F)}{t}=\tau^-(F).
    \]
Additionally, the following inequalities for the mass growth function hold:
\begin{align*}
    t\cdot \tau^+(F) & \leq h_{\sigma,t}(F) \leq h_{\sigma,0}(F) + t \cdot \tau^+(F) && \text{for }t\geq 0,\\
    t\cdot \tau^-(F) & \leq h_{\sigma,t}(F) \leq h_{\sigma,0}(F) + t \cdot \tau^-(F) && \text{for }t\leq 0.
\end{align*}
\end{theorem}

\begin{proof}
The proof follows the same idea as in \cite[Proposition~3.10, Lemma~3.12]{KikutaOuchiTakahashi}.
Let $\sigma$ be any Bridgeland stability condition on $\cD$.
By the support property (\autoref{def:stability}\ref{def:supportproperty}), there exists a constant $C'>0$ such that $|Z_\sigma(E)|>C'$ holds for any $\sigma$-semistable object $E$. Suppose that $t>0$. Then
    \[
    C'\cdot e^{\phi^+_\sigma(F^nG)t}\leq m_{\sigma,t}(F^nG)\leq
    m_{\sigma,0}(F^nG)\cdot e^{\phi^+_\sigma(F^nG)t}.
    \]
By applying $\frac{\log(\cdot)}{n}$ and taking $n\ra\infty$, one has
    \begin{equation}
    \label{eqn:stability_prelimit}
    t\cdot\limsup_{n\ra\infty}\frac{\phi^+_\sigma(F^nG)}{n}\leq h_{\sigma,t}(F)\leq
    h_{\sigma,0}(F)+t\cdot\limsup_{n\ra\infty}\frac{\phi^+_\sigma(F^nG)}{n}.
    \end{equation}
Note that $0\leq h_{\sigma,0}(F)\leq h_0(F)<\infty$ is a real number. By dividing $t$ and letting $t\ra\infty$, one obtains
    \[
    \lim_{t\ra\infty}\frac{h_{\sigma,t}(F)}{t}=
    \limsup_{n\ra\infty}\frac{\phi^+_\sigma(F^nG)}{n}.
    \]

On the other hand, recall that
\[
\ep^+(G,F^nG)=\max\{k\in\bZ\colon\Hom(G,F^nG[-k])\neq0\}.
\]
Hence
    \[
    \phi^+_\sigma(F^nG)-\phi^-_\sigma(G)\geq\ep^+(G,F^nG).
    \]
Thus
    \[
    \liminf_{n\ra\infty}\frac{\phi^+_\sigma(F^nG)}{n}\geq\lim_{n\ra\infty}\frac{\ep^+(G,F^nG)}{n}=\tau^+(F)=\lim_{t\ra\infty}\frac{h_t(F)}{t}
    \]
by \autoref{theorem:shiftingnumberviaentropy}.
Combining these inequalities together with \autoref{theorem:Ikedamassentropybound} gives
    \[
    \liminf_{n\ra\infty}\frac{\phi^+_\sigma(F^nG)}{n}\geq\tau^+(F)=\lim_{t\ra\infty}\frac{h_t(F)}{t}\geq\lim_{t\ra\infty}\frac{h_{\sigma,t}(F)}{t}=\limsup_{n\ra\infty}\frac{\phi^+_\sigma(F^nG)}{n}.
    \]
This proves the first part of the statement, namely that the limits exist and:
    \[
    \lim_{n\ra\infty}\frac{\phi^+_\sigma(F^nG)}{n}=\lim_{t\ra\infty}\frac{h_{\sigma,t}(F)}{t}=\tau^+(F).
    \]
Plug this into \autoref{eqn:stability_prelimit} one obtains
\begin{align*}
    t\cdot \tau^+(F) & \leq h_{\sigma,t}(F) \leq h_{\sigma,0}(F) + t \cdot \tau^+(F) && \text{for }t\geq 0.
\end{align*}

Now we prove the second part of the statement.
Assume that $t<0$, so we have now instead:
    \[
    C'\cdot e^{\phi^-_\sigma(F^nG)t}\leq m_{\sigma,t}(F^nG)\leq
    m_{\sigma,0}(F^nG)\cdot e^{\phi^-_\sigma(F^nG)t}.
    \]
Again by taking $\frac{\log(\cdot)}{n}$ and $\lim_{n\to +\infty}$ (and remembering $t<0$ now):
\begin{multline*}
    t\cdot\liminf_{n\ra\infty}\frac{\phi^-_\sigma(F^nG)}{n}=\limsup_{n\ra\infty}\frac{\phi^-_\sigma(F^nG)t}{n}\leq h_{\sigma,t}(F) \\ \leq
    h_{\sigma,0}(F)+\limsup_{n\ra\infty}\frac{\phi^-_\sigma(F^nG)t}{n}=h_{\sigma,0}(F)+t\cdot\liminf_{n\ra\infty}\frac{\phi^-_\sigma(F^nG)}{n}.
\end{multline*}
It follows that
    \[
    \lim_{t\ra-\infty}\frac{h_{\sigma,t}(F)}{t}=\liminf_{n\ra\infty}\frac{\phi^-_\sigma(F^nG)}{n}.
    \]

Recall that
\[
\ep^-(G,F^nG)=\min\{k\in\bZ\colon\Hom(G,F^nG[-k])\neq0\}.
\]
We assumed that $\cD$ admits a Serre functor $\bS$ so we have
\[
\Hom(F^nG[-\ep^-(G,F^nG)],\bS G)\neq0.
\]
Hence
\begin{align*}
    \phi^+_\sigma(\bS G)-\phi^-_\sigma(F^nG)&\geq-\ep^-(G,F^nG), \text{ or equivalently}\\
    \ep^-(G,F^nG) &\geq \phi^-_\sigma(F^nG) - \phi^+_\sigma(\bS G).
\end{align*}
Combined with \autoref{theorem:shiftingnumberviaentropy} this gives
    \[
    \lim_{t\ra-\infty}\frac{h_t(F)}{t}=\tau^-(F)=\lim_{n\ra\infty}\frac{\ep^-(G,F^nG)}{n}\geq\limsup_{n\ra\infty}\frac{\phi^-_\sigma(F^nG)}{n}.
    \]
Together with the bound $h_{\sigma,t}(F)\leq h_t(F)$ from \autoref{theorem:Ikedamassentropybound} (and recalling $t<0$ now) we find
    \[
    \liminf_{n\ra\infty}\frac{\phi^-_\sigma(F^nG)}{n}=\lim_{t\ra-\infty}\frac{h_{\sigma,t}(F)}{t}\geq\lim_{t\ra-\infty}\frac{h_t(F)}{t}=\tau^-(F)\geq\limsup_{n\ra\infty}\frac{\phi^-_\sigma(F^nG)}{n}.
    \]
This concludes the proof.
\end{proof}



\subsection{Further properties of shifting numbers}
	\label{ssec:further_properties_of_shifting_numbers}

We show that the basic properties of Poincar\'e translation numbers listed in the introduction are also satisfied by the shifting numbers of endofunctors.

\begin{proposition}
\label{proposition:basicproperties}
Let $F,G\colon\cD\ra\cD$ be endofunctors of a triangulated category $\cD$.
Then we have:
    \begin{enumerate}
        \item\label{proposition:compatiblewithshifts} $\tau^\pm(F\circ[k])=\tau^\pm(F)+k$ for any $k\in\bZ$. In particular, $\tau([k])=\tau^\pm([k])=k$.
        \item\label{proposition:compatiblewithpowers} $\tau^\pm(F^n)=n\tau^\pm(F)$ for any positive integer $n$.
        \item\label{proposition:inverse_tau} If $F$ is an autoequivalence and $\cD$ admits a Serre functor, then $\tau^\pm(F^{-1})=-\tau^\mp(F)$.
        \item\label{proposition:conjugacyinvariant_tau} $\tau^\pm(FG)=\tau^\pm(GF)$. In particular, if $G$ is an autoequivalence of $\cD$, then $\tau^\pm(GFG^{-1})=\tau^\pm(F)$.
    \end{enumerate}
\end{proposition}

\begin{proof}
These statements all follow straightforwardly from
    \[
    \tau^\pm(F)=\lim_{t\ra\pm\infty}\frac{h_t(F)}{t}
    \]
established in \autoref{theorem:shiftingnumberviaentropy}, and some basic properties of the categorical entropy function $h_t(F)$.

Parts (i) and (ii) follow from the formulas 
\[
	h_t(F\circ[k])=h_t(F)+kt\text{ and } h_t(F^n)=nh_t(F)
	\text{ for any }n\geq1
\]
established in \cite[\S2]{DimitrovHaidenKatzarkov_Dynamical-systems-and-categories} or \cite[Lemma~3.7]{KiTa}.
If $\cD$ admits a Serre functor, then $h_t(F^{-1})=h_{-t}(F)$ by \cite[Lemma~2.11]{FFO} and this proves part (iii).
Part (iv) follows from $h_t(FG)=h_t(GF)$, see e.g. \cite[Lemma~2.8]{Kikuta}.
\end{proof}

\noindent We will use the following lemma in the next section.

\begin{lemma}
\label{lem:preserveheartthenzero}
Let $F\colon\cD\ra\cD$ be an endofunctor of a triangulated category $\cD$. Suppose there is a split generator $G$ of $\cD$ and an integer $M\geq0$ such that
    \[
    \Hom(G,F^nG[k])=0 \text{ for any } |k|\geq M\text{ and } n\geq0,
    \]
(for instance, when $F$ preserves a bounded $t$-structure of finite cohomological dimension). Then $\tau(F)=\tau^\pm(F)=0$.
\end{lemma}

\begin{proof}
Under the vanishing assumption, $h_t(F)$ is a constant function in $t$ \cite[Lemma~2.11]{DimitrovHaidenKatzarkov_Dynamical-systems-and-categories}. Thus $\tau(F)=\tau^\pm(F)=0$ by \autoref{theorem:shiftingnumberviaentropy}.
\end{proof}



\subsection{Legendre transform of entropy functions}
	\label{ssec:LegendreTransform}
Recall from \autoref{thm:convexity_and_finiteness} that the categorical entropy function $h_t(F)$ of any endofunctor $F$ of a triangulated category is a real-valued convex function in the variable $t$. It is then natural to consider its \emph{Legendre transform}
\[
h^*(F)\colon I^*\ra\bR,
\]
where the domain is
\[
I^*\coloneqq\Big\{t^*\in\bR\colon\sup_{t\in\bR}\Big(t^*t-h_t(F)\Big)<\infty\Big\},
\]
and the value at $t^*\in I^*$, denoted by $h_{t^*}^*(F)$, is defined to be
\[
h_{t^*}^*(F)\coloneqq\sup_{t\in\bR}\Big(t^*t-h_t(F)\Big).
\]

\begin{proposition}
Let $F$ be an endofunctor of a triangulated category, and let $h^*(F)$ be the Legendre transform of the associated categorical entropy function $h_t(F)$. Then
    \begin{enumerate}
        \item\label{prop:Legendredomain} The domain of definition of $h^*(F)$ is $[\tau^-(F),\tau^+(F)]$.
        \item\label{prop:Legendreinf} The minimum of $h^*(F)$ is given by the categorical entropy:
        $$\min\{h^*_{t^*}(F)\colon t^*\in[\tau^-(F),\tau^+(F)]\}=-h_0(F).$$
        \item\label{prop:Legendre_shift}
        Applying a shift to the functor we have that the Legendre transform is also shifted:
        \[
            h^*_{t^*}\left(F\circ [k]\right) = h^{*}_{t^*-k}(F)
        \]
    \end{enumerate}
\end{proposition}
\begin{proof}
Part \ref{prop:Legendredomain} follows straightforwardly from \autoref{theorem:shiftingnumberviaentropy} and the definition of Legendre transform.
To prove \ref{prop:Legendreinf}, first observe that for any $t^*\in[\tau^-(F),\tau^+(F)]$, we have
\[
h^*_{t^*}(F)\geq t^*\cdot0-h_0(F)=-h_0(F).
\]
On the other hand, since $h_t(F)$ is convex, it has a left and right derivative at $0$ because the function
\[
\frac{h_t(F)-h_0(F)}{t}
\]
is increasing for $t\in\bR\backslash\{0\}$.
Hence for any $t^*\in\bR$ satisfying
	\[
	\lim_{t\to 0^-}\frac{h_t(F)-h_0(F)}{t}\leq t^*\leq\lim_{t\to 0^+}\frac{h_t(F)-h_0(F)}{t}
	\]
we have that $h^*_{t^*}(F)=-h_0(F)$ so the value is achieved.

Finally \ref{prop:Legendre_shift} follows from the property $h_t(F\circ [k]) = h_t(F) + kt$ established in \cite{DimitrovHaidenKatzarkov_Dynamical-systems-and-categories} and standard properties of the Legendre transform.
\end{proof}




\section{Examples of shifting numbers}
	\label{sec:examples_of_shifting_numbers}

\paragraph{Outline of section}
In this section, we compute the shifting numbers of the standard autoequivalences of $\cD^b\Coh(X)$, spherical twists, $\bP$-twists, pseudo-Anosov autoequivalences in the sense of \cite{DimitrovHaidenKatzarkov_Dynamical-systems-and-categories}, and an autoequivalence of a Calabi--Yau category that is pseudo-Anosov in a more general sense proposed in \cite{FFHKL}.


\subsection{Standard autoequivalences}
	\label{ssec:standard_autoequivalences}

\subsubsection{Setup}
	\label{sssec:setup_standard_autoequivalences}
Let $X$ be a smooth projective variety over the base field $\tbk$ and $\cD=\cD^b\Coh(X)$ be the bounded derived category of coherent sheaves on $X$. The group of \emph{standard autoequivalences} of $\cD$ is the subgroup of $\Aut(\cD)$ defined by:
    \[
    \Aut_\std(\cD)\coloneqq(\Aut(X)\ltimes\Pic(X))\times\bZ[1]\subseteq\Aut(\cD).
    \]

\begin{theorem}
\label{theorem:standardautoequivalence}
We have an agreement of upper and lower shifting numbers $\tau(F)=\tau^\pm(F)\in\bR$ for any $F\in\Aut_\std(\cD)$.
Moreover, the map $\tau\colon\Aut_\std(\cD)\ra\bR$ given by the shifting number is a group homomorphism and can be factored as
    \[
    \Aut_\std(\cD)\xrightarrow{\pi}\bZ\xrightarrow{\iota}\bR,
    \]
where $\pi$ is the projection to the $\bZ[1]$-factor, and $\iota$ is the natural inclusion of integers into real numbers.
\end{theorem}

\begin{proof}
Let $F=\bL f^*(-\otimes L)[n]\in\Aut_\std(\cD)$, where $f\in\Aut(X)$, $L\in\Pic(X)$, and $n\in\bZ$.
Since $F\circ[-n]=\bL f^*(-\otimes L)$ preserves the standard t-structure on $\cD$, we have $\tau^\pm(F\circ[-n])=0$ by \autoref{lem:preserveheartthenzero}. Hence $\tau^\pm(F)=n$ by \autoref{proposition:basicproperties}\ref{proposition:compatiblewithshifts}.
\end{proof}

\begin{corollary}
\label{corollary:standardautoequivalence}
Let $X$ be a smooth projective variety such that $K_X$ is ample or anti-ample. Then the shifting numbers give a homomorphism
    \[
    \tau\colon\Aut(\cD^b\Coh(X))\ra\bR.
    \]
\end{corollary}

\begin{proof}
By \cite[Theorem~3.1]{BondalOrlov2001_Reconstruction-of-a-variety-from-the-derived-category-and-groups-of-autoequivalences} we have the equality $\Aut(\cD^b\Coh(X))=\Aut_\std(\cD^b\Coh(X))$ if $K_X$ is ample or anti-ample.
The result then follows from \autoref{theorem:standardautoequivalence}.
\end{proof}



\subsection{Spherical twists and \texorpdfstring{$\bP$}{bP}-twists}
	\label{subsec:sphericaltwistsPtwists}

\subsubsection{Setup}
	\label{sssec:setup_spherical_twists_and_bp_twists}
Let $\cD$ be a triangulated category which admits a Serre functor $\bS$ and let $N$ be a positive integer.
Recall that an object $E$ in $\cD$ is called \emph{$N$-spherical} if $\bS(E)\cong E[N]$ and $\Hom^\bullet(E,E)\cong\tbk\oplus\tbk[-N]$.
Examples of spherical objects include line bundles in the bounded derived categories of $N$-dimensional Calabi--Yau manifolds, and Lagrangian spheres in certain derived Fukaya categories \cite{SeidelThomas_Braid-group-actions-on-derived-categories-of-coherent}.
Seidel and Thomas \cite{SeidelThomas_Braid-group-actions-on-derived-categories-of-coherent} introduce an autoequivalence of $\cD$ associated to each spherical object $E$, which is called the \emph{spherical twist} $T_E$.
It is defined by
    \[
    T_E(F)\coloneqq\Cone(\Hom^\bullet(E,F)\otimes E\xrightarrow{\mathrm{ev}} F),
    \]
and is the categorical analogue of Dehn twists along Lagrangian spheres.

\begin{theorem}[Spherical twist calculation]
    \label{thm:spherical_twist_calculation}
	Let $E$ be an $N$-spherical object in $\cD$ and suppose that $E^\perp\coloneqq\{F\in\cD\colon\Hom^\bullet(E,F)=0\}\neq\emptyset$.
	Then
	    \[
	    \tau^+(T_E)=0 \text{ \ and \ } \tau^-(T_E)=1-N.
	    \]
	Hence $\tau(T_E)=\frac{1-N}{2}$.
\end{theorem}

\begin{proof}
This follows directly from \autoref{theorem:shiftingnumberviaentropy} and \cite[Theorem~3.1]{Ouchi2020_On-entropy-of-spherical-twists} which states that
    \[
    h_t(T_E)=
        \begin{cases}
        (1-N)t, & \text{if }t\leq0;\\
        0, & \text{if }t\geq0.
        \end{cases}
    \]
\end{proof}

\begin{remark}
In general, it is difficult to compute the categorical entropy function, and therefore the shifting numbers, of a composition of several spherical twists $T_{E_1},\ldots,T_{E_n}$ with non-trivial couplings (i.e.~$\Hom^\bullet(E_i,E_j)\neq0$).
We work out the details in \autoref{sec:CYA2} of the case of autoequivalences of Calabi--Yau category associated to the $A_2$ quiver.
This is the simplest example which contains non-trivial couplings of spherical twists.
\end{remark}

\subsubsection{\texorpdfstring{$\bP$}{bP}-twists}
	\label{sssec:bp_twists}
Huybrechts and Thomas \cite{HuybrechtsThomas2006_Bbb-P-objects-and-autoequivalences-of-derived-categories} introduced the categorical analogue of a Dehn twist along a Lagrangian complex projective plane.
Recall that an object $E$ is called a \emph{$\bP^N$-object} if $\bS(E)\cong E[2N]$ and $\Hom^\bullet(E,E)\cong H^*(\bP^N,\bZ)\otimes\tbk$ as $\tbk$-algebras.
Examples of $\bP^N$-objects include line bundles and the structure sheaf of an embedded $\bP^N$ in the bounded derived categories of a $2N$-dimensional holomorphic symplectic manifold.
One can also define an autoequivalence, called the $\bP$-twist $P_E$, associated to an $\bP^N$-object $E$.
We refer to \cite{HuybrechtsThomas2006_Bbb-P-objects-and-autoequivalences-of-derived-categories} for the precise definition.

\begin{proposition}
	Let $E$ be a $\bP^N$ object in $\cD$ and suppose that $E^\perp\coloneqq\{F\in\cD\colon\Hom^\bullet(E,F)=0\}\neq\emptyset$. Then
	    \[
	    \tau^+(P_E)=0 \text{ \ and \ } \tau^-(P_E)=-2N.
	    \]
	Hence $\tau(P_E)=-N$.
\end{proposition}

\begin{proof}
This follows directly from \autoref{theorem:shiftingnumberviaentropy} and \cite[Theorem~3.1]{2018arXiv180110485F} which states that
    \[
    h_t(T_E)=
        \begin{cases}
        -2Nt, & \text{if }t\leq0;\\
        0, & \text{if }t\geq0.
        \end{cases}
    \]
\end{proof}



\subsection{Pseudo-Anosov autoequivalences}
	\label{subsec:pseudoAnosov}
The notion of \emph{pseudo-Anosov autoequivalences} of triangulated categories was introduced in \cite{DimitrovHaidenKatzarkov_Dynamical-systems-and-categories} as a categorical generalization of pseudo-Anosov maps on Riemann surfaces. 
To formulate the definition, recall that the space of Bridgeland stability conditions $\Stab(\cD)$ carries natural group actions, by $\Aut(\cD)$ on the left and by $\glstab$ on the right, see \cite[Lemma~8.2]{Bridgeland2007_Stability-conditions-on-triangulated-categories}.

To define the action, for an autoequivalence $F\in\Aut(\cD)$ set
    \[
    \sigma=(Z,P)\mapsto F\cdot\sigma\coloneqq(Z\circ F^{-1},P'),
    \]
where $P'(\phi)\coloneqq F(P(\phi))$.
To define the action of $\glstab$ on $\Stab(\cD)$, recall that $\glstab$ is isomorphic to the group of pairs $(T,f)$, where $T\in\GL^+_2(\bR)$ and $f\colon\bR\ra\bR$ is an increasing map with $f(\phi+1)=f(\phi)+1$, such that their induced maps on $S^1\cong(\bR^2\backslash\{(0,0)\})/\bR_{>0}\cong\bR/2\bZ$ coincide.
For $g=(T,f)\in\glstab$ define
    \[
    \sigma=(Z,P)\mapsto\sigma\cdot g\coloneqq(T^{-1}\circ Z,P''),
    \]
where $P''(\phi)\coloneqq P(f(\phi))$.
It can be checked that the actions of $\Aut(\cD)$ and $\glstab$ commute.

\subsubsection{Rotation number for $\glstab$}
    \label{sssec:rotation_number_for_glstab}
We have the following diagram:
\begin{equation*}
    \label{eqn_cd:gltilde_homeoR}
\begin{tikzcd}
    0
    \arrow[r]
    &
    \bZ
    \arrow[r, hook, ""]
    \arrow[d, equal, ""]    
    &
    \glstab
    \arrow[d, hook, ""]
    \arrow[r, twoheadrightarrow, ""]
    &\GL_2^+(\bR)
    \arrow[d, hook, ""]
    \arrow[r]
    &
    \id
    \\
    0
    \arrow[r]
    &
    \bZ
    \arrow[r, hook, ""]
    \arrow[d, equal, ""]
    & 
    \Homeo^+_{\bZ}(\bR)
    \arrow[r, twoheadrightarrow, ""]
    \arrow[d, "\wtilde{\rho}"]
    &
    \Homeo^+(\bR/\bZ)
    \arrow[r]
    \arrow[d, "\rho"]
    &
    \id\\
    &
    \bZ
    \arrow[r]
    &
    \bR
    \arrow[r]
    &
    \bR/\bZ
\end{tikzcd}
\end{equation*}
The first two rows are exact sequences (and central extensions) and the maps from the middle to the last row are given by the \Poincare translation and rotation numbers respectively.
This observation provides a connection between the Poincar\'e translation number and the shifting numbers of autoequivalences in certain examples, see \autoref{proposition:DHKKrotationnumber} and \autoref{sec:ellcurve}.

Now we recall the definition of pseudo-Anosov autoequivalences from \cite{DimitrovHaidenKatzarkov_Dynamical-systems-and-categories}.

\begin{definition}[{\cite[Definition~4.1]{DimitrovHaidenKatzarkov_Dynamical-systems-and-categories}}]
\label{definition:DHKKpseudoAnosov}
An autoequivalence $F\in\Aut(\cD)$ is said to be \emph{pseudo-Anosov} if there exists a Bridgeland stability condition $\sigma\in\Stab(\cD)$ and an element $g=(T,f)\in\glstab$ such that
    \begin{enumerate}
        \item $F\cdot\sigma=\sigma\cdot g$,
        \item $T=\begin{pmatrix}r&0\\0&r^{-1}\end{pmatrix}\text{ or }\begin{pmatrix}r^{-1}&0\\0&r\end{pmatrix}$ for some $\lambda\coloneqq|r|>1$.
    \end{enumerate}
\end{definition}

\begin{remark}[On pseudo-Anosov autoequivalences]
	\label{remark:on_pseudo_anosov_autoequivalences}
	\leavevmode
	\begin{enumerate}
		\item One obtains an equivalent definition if only requiring $T$ to be \emph{conjugate} to a diagonal matrix, instead of being equal.
		Indeed, if $F\cdot \sigma = \sigma \cdot cgc^{-1}$ then $F\cdot (\sigma \cdot c) = (\sigma \cdot c)\cdot g $.
		\item The stability condition $\sigma$ is analogous to a pair of measured foliations on a Riemann surface, and condition (i) and (ii) describe the expansion/contraction of the foliations by a pseudo-Anosov map.
		\item Examples of pseudo-Anosov autoequivalences include the Serre functor of the derived category of representations of Kronecker quiver with at least three arrows \cite[\S4.2]{DimitrovHaidenKatzarkov_Dynamical-systems-and-categories}, and autoequivalences on the bounded derived category of an elliptic curve such that their induced actions on the numerical Grothendieck group are hyperbolic \cite[Proposition~4.14]{KikutaCurvature}.
		Note that there is a more general definition of pseudo-Anosov autoequivalences introduced in \cite{FFHKL}, where examples of pseudo-Anosov autoequivalences of certain $3$-Calabi--Yau categories (with respect to the more general definition) are provided.
	\end{enumerate}
\end{remark}

\begin{proposition}
\label{proposition:DHKKrotationnumber}
Let $\cD$ be a triangulated category with a split generator $G$, and let $F$ be a pseudo-Anosov autoequivalence of $\cD$ in the sense of \autoref{definition:DHKKpseudoAnosov}.
Then the upper and lower shifting numbers agree and satisfy
    \[
    \tau(F)=\tau^\pm(F)=\wtilde{\rho}(f)=f(0)\in\bZ,
    \]
where $\wtilde{\rho}(f)$ is the Poincar\'e translation number of $f\in\Homeo^+_\bZ(\bR)$ defined in the introduction.
\end{proposition}

\begin{proof}
Since $F\cdot\sigma=\sigma\cdot g$, we have $F(P_\sigma(\phi))=P_\sigma(f(\phi))$ for any phase $\phi\in\bR$. Let $A_1,\ldots,A_n$ be the $\sigma$-Harder--Narasimhan semistable factors of a split generator $G$ with phases $\phi_\sigma(A_1)>\cdots>\phi_\sigma(A_n)$. Then
    \[
    F(A_i)\in F(P_\sigma(\phi(A_i)))=P_{\sigma}(f(\phi(A_i))).
    \]
Hence $F(A_1),\ldots,F(A_n)$ are the $\sigma$-Harder--Narasimhan factors of $F(G)$ with phases $f(\phi(A_1))>\cdots>f(\phi(A_n))$. Therefore, we have
    \[
    \phi_\sigma^+(F^kG)=f^{(k)}(\phi(A_1))\text{ \ and \ }
    \phi_\sigma^-(F^kG)=f^{(k)}(\phi(A_n))
    \]
for any $k\geq1$. Hence
    \[
    \tau(F)=\tau^\pm(F)=\wtilde{\rho}(f)
    \]
by \autoref{theorem:computeviastability} and the definition of Poincar\'e translation number.

By condition (ii) in \autoref{definition:DHKKpseudoAnosov}, the $x$-axis and $y$-axis are preserved under the linear map $T$. Since $T$ and $f$ are compatible under the identification $S^1\cong(\bR^2\backslash\{(0,0)\})/\bR_{>0}\cong\bR/2\bZ$, we have $f(0)\in\bZ$.
Thus $f^{(k)}(0)=kf(0)$ since $f$ is compatible with integral shifts.
Therefore $\wtilde{\rho}(f)=f(0)\in\bZ$.
\end{proof}



\subsection{An example for Calabi--Yau manifolds}
	\label{subsec:example_CY}
We include also a computation of shifting numbers for pseudo-Anosov autoequivalences associated to Calabi--Yau manifolds.
\begin{theorem}
	\label{theorem:example_CY}
Let $X$ be a projective Calabi--Yau manifold of dimension $N\geq3$. Consider the autoequivalence of $\cD^b\Coh(X)$ given by the composition
	\[
	F\coloneqq T_{\cO_X}\circ(-\otimes\cO_X(-1)).
	\]
Then
	 \[
	 \tau^+(F)=0 \text{ \ and \ } \tau^-(F)=1-N.
	 \]
\end{theorem}

\begin{proof}
By \cite[Theorem~1.1]{Fan2018_Entropy-of-an-autoequivalence-on-Calabi-Yau-manifolds}, the value of the categorical entropy function $h_t(F)$ at $t\in\bR$ is the unique positive real number satisfying
	\[
	\sum_{k\geq1}\frac{\chi(\cO(k))}{e^{h_t(F)\cdot k}}=e^{(N-1)t}.
	\]
Note that $\chi(\cO(k))$ is a polynomial of degree $N$ in $k$ by Riemann--Roch. Recall also that for $\ell\in\bZ_{\geq0}$ and $x>1$, there is an integral polynomial $P_\ell$ of degree $\ell$ such that
 	\[
	\sum_{k\geq1}\frac{k^\ell}{x^k}=\mathrm{Li}_{-\ell}\Big(\frac{1}{x}\Big)=\frac{P_\ell(x)}{(x-1)^{\ell+1}}.
	\]
Hence we have
	\begin{equation}
		\label{equation:example_CY}
	e^{(N-1)t}=\sum_{k\geq1}\frac{\chi(\cO(k))}{e^{h_t(F)\cdot k}}=\frac{Q_N(e^{h_t(F)})}{(e^{h_t(F)}-1)^{N+1}}
	\end{equation}
for a polynomial $Q_N$ of degree $N$.

To compute the shifting numbers of $F$, we need to study the asymptotic behavior of $h_t(F)$ as $t\ra\pm\infty$.
First, let us consider $t\ra+\infty$.
By \autoref{theorem:shiftingnumberviaentropy} and the fact that $h_t(F)>0$ for any $t$ \cite[Theorem~1.1]{Fan2018_Entropy-of-an-autoequivalence-on-Calabi-Yau-manifolds} (for this particular autoequivalence $F$), the limit $\lim_{t\ra+\infty}h_t(F)$ is either zero or $+\infty$. 
Observe that the case $\lim_{t\ra+\infty}h_t(F)=+\infty$ can be excluded by \autoref{equation:example_CY}. Therefore we have $\lim_{t\ra+\infty}h_t(F)=0$, hence $\tau^+(F)=0$ by \autoref{theorem:shiftingnumberviaentropy}.

Now we let $t\ra-\infty$. By the same argument as above, the limit $\lim_{t\ra-\infty}h_t(F)$ is either zero or $+\infty$. In this case $\lim_{t\ra-\infty}h_t(F)=0$ is excluded by \autoref{equation:example_CY}, so we have $\lim_{t\ra-\infty}h_t(F)=+\infty$. By multiplying $e^{h_t(F)}$ on both sides of  \autoref{equation:example_CY}, we get
	\[
	e^{h_t(F)+(N-1)t}=\frac{e^{h_t(F)}Q_N(e^{h_t(F)})}{(e^{h_t(F)}-1)^{N+1}}.
	\]
Since $\lim_{t\ra-\infty}h_t(F)=+\infty$, the limit $\lim_{t\ra-\infty}$ of the right hand side of the above equality is a finite number. Therefore $\tau^-(F)=1-N$ by \autoref{theorem:shiftingnumberviaentropy}.
\end{proof}

Below is a sketch comparing the graphs of the entropy functions $h_t(T_{\mathcal{O}_X})$ and $h_t(T_{\cO_X}\circ(-\otimes\cO_X(-1)))$.

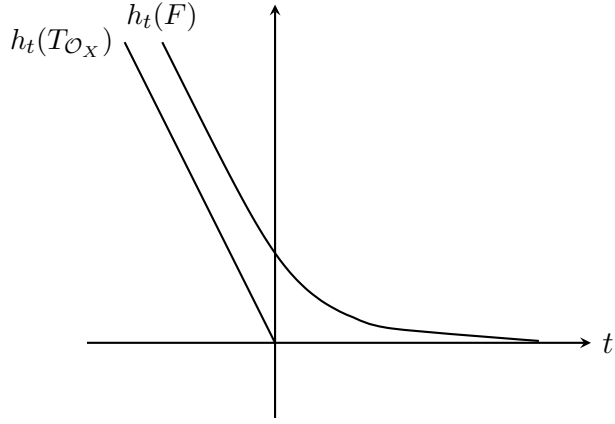
\begin{figure}[htbp]\centering
\begin{tikzpicture}[yscale=1,xscale=1,>=stealth]
  \draw[->,thick] (-2.5, 0) -- (4.2, 0) node[right] {$t$};
  \draw[->,thick] (0, -1) -- (0, 4.5);
  \draw[scale=0.5, domain=0.05:8, smooth, variable=\y,thick] plot ({(2+3*\y-4*ln(exp(\y)-1))/2}, {\y}) node[above] {\small $h_t(F)$};
  \draw[-,thick] (0,0) -- (-2,4) node[left] {\small $h_t(T_{\mathcal{O}_X})$};
\end{tikzpicture}
\end{figure}

\begin{remark}
The autoequivalence $F=T_{\cO_X}\circ(-\otimes\cO_X(-1))$, at least in the case of quintic Calabi--Yau threefolds, is a \emph{pseudo-Anosov autoequivalence} in the sense of \cite{FFHKL}.
It was used in \cite{Fan2018_Entropy-of-an-autoequivalence-on-Calabi-Yau-manifolds} to construct a counterexample to the categorical Gromov--Yomdin conjecture.
\end{remark}




\section{Quasimorphisms and shifting numbers}
	\label{sec:quasimorphisms_and_shifting_numbers}

\paragraph{Outline}
We establish in this section that in some cases of interest, the shifting number yields a quasimorphism on the corresponding group of autoequivalences.
Specifically, \autoref{sec:ellcurve} deals with the case of elliptic curves and \autoref{sec:abelsurface} deals with the case of an abelian surface.
In \autoref{sec:CYA2} the case of the $N$-Calabi--Yau category of the $A_2$ quiver is analyzed.
This last example is an instance where the upper and lower shifting numbers do not agree, but their average does give a quasimorphism.
The case of abelian surfaces is revisited in \autoref{sec:quasimorphisms_on_lie_groups}, where the same quasimorphism is obtained from a construction associated to the central extension of the Lie group $\SO_{2,\rho}(\bR)$.


\subsection{The derived category of an elliptic curve}
    \label{sec:ellcurve}

The case of autoequivalences for a curve of genus $g\geq 2$ or $g=0$ is handled by \autoref{corollary:standardautoequivalence}, and in \autoref{theorem:ellcurvequasihomom} below we handle the case of an elliptic curve.
Together, these results show that \autoref{ques:quasimorphism} has an affirmative answer if $\cD$ is the bounded derived category of coherent sheaves on a curve.

\begin{theorem}
\label{theorem:ellcurvequasihomom}
Let $\cD$ be the bounded derived category of coherent sheaves on an elliptic curve. Then $\tau^+(F)=\tau^-(F)\in\bR$, and the shifting numbers
    \[
    \tau=\tau^\pm\colon\Aut(\cD)\ra\bR
    \]
give a quasimorphism.
Moreover, $\tau$ can be factored as
    \[
    \Aut(\cD)\xrightarrow{s}\glstab\xrightarrow{t}\Homeo^+_\bZ(\bR)\xrightarrow{\wtilde{\rho}}\bR,
    \]
where $s,t$ are homomorphisms, and $\wtilde{\rho}$ is the quasimorphism given by the Poincar\'e translation number (see \autoref{sssec:rotation_number_for_glstab}).
\end{theorem}

\begin{proof}
First, we define the group homomorphism $s$. By \cite[Theorem~9.1]{Bridgeland2007_Stability-conditions-on-triangulated-categories}, the $\glstab$-action on $\Stab(\cD)$ is free and transitive. Fix a stability condition $\sigma_0\in\Stab(\cD)$, and define
    \[
    \Phi_{\sigma_0}\colon\glstab \toisom \Stab(\cD),\ \ g\mapsto\sigma_0\cdot g.
    \]
Using this identification, one defines a map
    \[
    s\colon \Aut(\cD)\ra\glstab,\ \ F\mapsto\Phi_{\sigma_0}^{-1}(F\cdot\sigma_0).
    \]
Since the $\Aut(\cD)$-actions and $\glstab$-actions on $\Stab(\cD)$ commute with each other, the map $s$ is a group homomorphism.

Next, we define the second group homomorphism $t$. Recall that elements in $\glstab$ can be represented by pairs $(T,f)$, where $T\in\GL^+(2,\bR)$ and $f\in\Homeo_\bZ(\bR)$ satisfy certain compatibility conditions. The group homomorphism $t$ is defined to be the map given by the second component
    \[
    t\colon\glstab\ra\Homeo^+_\bZ(\bR),\ \ 
    (T,f)\mapsto f.
    \]
    
Now we compute the shifting numbers of an autoequivalence $F\in\Aut(\cD)$. 
Choose a split generator $G=\cO\oplus\cO(1)$, where $\cO(1)$ is an ample line bundle on the elliptic curve.
By \autoref{theorem:computeviastability}, we have
    \[
    \tau^\pm(F)=\lim_{n\ra\infty}\frac{\phi^\pm_{\sigma_0}(F^nG)}{n}.
    \]
Recall that both $\cO$ and $\cO(1)$ are stable with respect to any stability condition on $\cD$. Denote $\phi_0\coloneqq\phi_{\sigma_0}(\cO)$ and $\phi_1\coloneqq\phi_{\sigma_0}(\cO(1))$. Then
    \[
    \phi_0=\phi_{F\cdot\sigma_0}(F(\cO))=\phi_{\sigma_0\cdot s(F)}(F(\cO)).
    \]
Hence $F(\cO)$ is $\sigma_0$-stable and of phase $((t\circ s)(F))(\phi_0)$. Denote the image of $F$ under $t\circ s$ by
    \[
    \widetilde{F}\coloneqq(t\circ s)(F)\in\Homeo^+_\bZ(\bR).
    \]
Then we have
    \[
    \phi_{\sigma_0}(F^n(\cO))=\widetilde{F}^{(n)}(\phi_0),\text{ and similarly }
    \phi_{\sigma_0}(F^n(\cO(1)))=\widetilde{F}^{(n)}(\phi_1).
    \]
Therefore
    \[
    \phi^+_{\sigma_0}(F^nG)=\widetilde{F}^{(n)}(\max\{\phi_0,\phi_1\})\text{ \ and \ }
    \phi^-_{\sigma_0}(F^nG)=\widetilde{F}^{(n)}(\min\{\phi_0,\phi_1\}).
    \]
This proves
    \[
    \tau(F)=\tau^\pm(F)=\wtilde{\rho}(\widetilde{F})=(\wtilde{\rho}\circ t\circ s)(F).
    \]
by \autoref{theorem:computeviastability}. Since the composition of a group homomorphism and a quasimorphism is a quasimorphism, this concludes the proof.
\end{proof}



\subsection{The derived category of an abelian surface}
    \label{sec:abelsurface}

\begin{theorem}
\label{theorem:abelsurface}
Let $\cD=\cD^b\Coh(X)$ be the bounded derived category of coherent sheaves on an abelian surface $X$. Then $\tau^+(F)=\tau^-(F)$ and the shifting numbers
    \[
    \tau=\tau^\pm\colon\Aut(\cD)\ra\bR
    \]
give a quasimorphism.

Furthermore, the shifting numbers agree with the homogenization of the quasimorphism
\begin{align*}
    \wtilde{\tau} \colon & \Aut(\cD)\to \bR\\
    & \wtilde{\tau}(F)  = \phi_{\sigma_0}(F(k(x)))
\end{align*}
where $\sigma_0$ is a fixed (geometric) stability condition and $\phi_{\sigma_0}$ denotes the phase with respect to $\sigma_0$, while $x\in X(\bC)$ is a point and $k(x)$ is the skyscraper sheaf at $x$.
\end{theorem}

\subsubsection{Distinguished component and slice}
    \label{sssec:distinguished_component_and_slice}
We first recall the description of a distinguished connected component $\Stab^\dagger(\cD)\subseteq\Stab(\cD)$ of the space of stability conditions on $\cD$.
The main reference is \cite{Bridgeland2008_Stability-conditions-on-K3-surfaces}.
For each complexified ample class $\beta+i\omega$, i.e.~$\beta,\omega\in\NS(X)\otimes\bR$ and $\omega$ ample, there is an associated stability condition $\sigma_{\beta,\omega}\in\Stab^\dagger(\cD)$.
The central charge of $\sigma_{\beta,\omega}$ is given by
    \[
    Z_{\beta,\omega}(E)=\left<\exp(\beta+i\omega),v(E)\right>,
    \]
where $\left<-,-\right>$ and $v(-)$ denote the Mukai pairing and Mukai vector, respectively.
Furthermore, skyscraper sheaves are $\sigma_{\beta,\omega}$-stable and of phase $1$, a fact particular to abelian surfaces.
This collection of stability conditions associated to complexified ample classes on $X$ defines a submanifold $V(\cD)\subseteq\Stab^\dagger(\cD)$.

\begin{theorem}[Bridgeland {\cite[\S15]{Bridgeland2008_Stability-conditions-on-K3-surfaces}}]
\label{theorem:bridgelandk3}
Let $\cD$ be the bounded derived category of coherent sheaves on an abelian surface, and consider $V(\cD)\subseteq\Stab^\dagger(\cD)$ as above.
    \begin{enumerate}
        \item The set $V(\cD)$ defines a slice for the $\glstab$-action, i.e. for any $\sigma\in\Stab^\dagger(\cD)$, there exists a unique element $g\in\glstab$ such that $\sigma\cdot g\in V(\cD)$.
        \item Autoequivalences preserve the distinguished connected component $\Stab^\dagger(\cD)$, i.e. for any $\sigma\in\Stab^\dagger(\cD)$ and $F\in\Aut(\cD)$, $F\cdot\sigma\in\Stab^\dagger(\cD)$, 
    \end{enumerate}
\end{theorem}

Now we prove \autoref{theorem:abelsurface}.

\begin{proof}[Proof of \autoref{theorem:abelsurface}]
By \cite{ArcaraMiles} line bundles are stable with respect to stability conditions in $V(\cD)$ and therefore stable with respect to stability conditions in $\Stab^\dagger(\cD)$ by \autoref{theorem:bridgelandk3}(i) since $V(\cD)$ is a slice for the $\glstab$-action.
By \autoref{theorem:bridgelandk3}(ii) the images of line bundles or skyscraper sheaves under autoequivalences are also stable with respect to any stability condition in $\Stab^\dagger(\cD)$.

Fix now a stability condition $\sigma_0\in V(\cD)$ and a split generator $G=\cO\oplus\cO(1)\oplus\cO(2)$.
We would like to compute the shifting numbers via \autoref{theorem:computeviastability}:
    \[
    \tau^\pm(F)=\lim_{n\ra\infty}\frac{\phi^\pm_{\sigma_0}(F^nG)}{n}.
    \]
For any line bundle $L$ and skyscraper sheaf $k(x)$ on $X$, one has
    \[
    \Hom(L,k(x))\cong\Hom(k(x),L[2])^\vee\neq0.
    \]
Therefore
    \[
    \phi_{\sigma_0}(F^nk(x))-2\leq\phi_{\sigma_0}(F^nL)\leq\phi_{\sigma_0}(F^nk(x))
    \]
for any autoequivalence $F$ and any $n\in\bZ$. Hence
    \[
    \phi_{\sigma_0}(F^nk(x))-2\leq\phi^-_{\sigma_0}(F^nG)\leq\phi^+_{\sigma_0}(F^nG)\leq\phi_{\sigma_0}(F^nk(x)).
    \]
By dividing $n$ and taking $n\ra\infty$, one obtains
    \[
    \limsup_{n\ra\infty}\frac{\phi_{\sigma_0}(F^nk(x))}{n}\leq\tau^-(F)\leq\tau^+(F)\leq\liminf_{n\ra\infty}\frac{\phi_{\sigma_0}(F^nk(x))}{n}.
    \]
Thus the limit $\lim_{n\ra\infty}\Big(\phi_{\sigma_0}(F^nk(x))/n\Big)$ exists and we have agreement of upper and lower shifting numbers:
    \[
    \tau(F)=\tau^\pm(F)=\lim_{n\ra\infty}\frac{\phi_{\sigma_0}(F^nk(x))}{n}.
    \]
    
Observe that $\tau$ is the homogenization of
    \[
    \widetilde{\tau}\colon\Aut(\cD)\ra\bR,\ \ F\mapsto\phi_{\sigma_0}(F(k(x))).
    \]
Therefore, to prove the theorem, it suffices to show that $\widetilde{\tau}$ is a quasimorphism, i.e.~there exists a constant $C$ such that for any $F,G\in\Aut(\cD)$,
    \[
    |\widetilde{\tau}(FG)-\widetilde{\tau}(F)-\widetilde{\tau}(G)|\leq C.
    \]
One can assume that
    \[
    0<\phi_{\sigma_0}(F(k(x)))\leq1\text{ \ and \ } 0<\phi_{\sigma_0}(G(k(x)))\leq1
    \]
by composing $F$ and $G$ with appropriate powers of the shift functor $[1]\in\Aut(\cD)$.
It remains to find an uniform bound of $\phi_{\sigma_0}(FG(k(x)))$ under these conditions.

Let $g=(T,f)\in\glstab$ be the unique element such that $F^{-1}\sigma_0g\in V(\cD)$. Then
    \[
    f(1)=\phi_{F^{-1}\sigma_0}(k(x))=\phi_{\sigma_0}(F(k(x)))\in(0,1]
    \]
by our assumption.
The quantity we would like to bound is
    \[
    \phi_{\sigma_0}(FG(k(x)))=\phi_{F^{-1}\sigma_0}(G(k(x)))=f(\phi_{F^{-1}\sigma_0g}(G(k(x)))).
    \]

There exists an integer $n$ such that $G(k(x))[n]$ is a coherent sheaf on $X$, since any Fourier--Mukai transform between derived categories of abelian surfaces is a sheaf transform \cite[Corollary~2.10]{BridgelandMaciocia}.
Then by \cite[Lemma~10.1(c)]{Bridgeland2008_Stability-conditions-on-K3-surfaces},
    \[
    -1<\phi_{\sigma_0}(G(k(x))[n])\leq1.
    \]
Since we assumed that $0<\phi_{\sigma_0}(G(k(x)))\leq1$ it follows that $-1\leq n\leq0$.

Since $F^{-1}\sigma_0g\in V(\cD)$, again by \cite[Lemma~10.1(c)]{Bridgeland2008_Stability-conditions-on-K3-surfaces} we have 
    \[
    -1<\phi_{F^{-1}\sigma_0g}(G(k(x))[n])\leq1.
    \]
Hence
    \[
    -2<\phi_{F^{-1}\sigma_0g}(G(k(x)))\leq1.
    \]
Since $f$ is an increasing map and compatible with integral shifts, using $f(1)\in(0,1]$ we obtain
    \[
    -3<f(-2)\leq\phi_{\sigma_0}(FG(k(x)))=f(\phi_{F^{-1}\sigma_0g}(G(k(x))))\leq f(1)\leq1.
    \]
This concludes the proof.
\end{proof}

\subsection{The Calabi--Yau categories of the \texorpdfstring{$A_2$}{A2} quiver}
    \label{sec:CYA2}

\subsubsection{Setup}
    \label{sssec:setup_the_calabi_yau_categories_of_the_a_2_quiver}
Fix for this section a positive integer $N\geq3$ and consider the bounded derived category $\cD_N=\cD_N(A_2)$ of the $N$-Calabi--Yau Ginzburg dg-algebra associated to the $A_2$ quiver $(\bullet\ra\bullet)$.
It is one of the simplest triangulated categories admitting spherical objects with non-trivial intersections.
We refer to \cite{Ginzburg2006_Calabi-Yau-algebras,Keller} for the precise definition.
The triangulated category $\cD_N$ is characterized by the following properties:
\begin{enumerate}
    \item It is an $N$-Calabi--Yau category, i.e. we have isomorphism
    $$
    \Hom^\bullet(E,F)\cong\Hom^\bullet(F,E[N])^\vee \text{ for any }
    E,F\in\cD_N.
    $$
    \item It is generated by two $N$-spherical objects $S_1$ and $S_2$ satisfying
    \[
        \Hom^\bullet(S_1,S_2)=\bC[-1].
    \]
\end{enumerate}
The space of Bridgeland stability conditions of $\cD_N$ is studied in \cite{BQS20}.
There is a distinguished connected component $\Stab_*(\cD_N)\subseteq\Stab(\cD_N)$ containing stability conditions whose heart $P(0,1]$ coincides with the canonical heart $\cH_N=\left<S_1,S_2\right>_\ext$ (see \cite[\S2 and \S4]{BQS20}).

Consider the subgroup $\Aut_*(\cD_N)\subseteq\Aut(\cD_N)$ generated by the shift functor $[1]$ and the spherical twists $T_1,T_2$ associated to the spherical objects $S_1,S_2$.
By \cite[Proposition~2.7]{BQS20}, the subgroup can be described using generators and relations as
    \[
    \Aut_*(\cD_N)=\left<T_1,T_2,[1]\ \Big|\ T_1T_2T_1=T_2T_1T_2,\ (T_1T_2)^3=[4-3N],\ T_i[1]=[1]T_i\right>.
    \]

\subsubsection{Central extension and braid group}
    \label{sssec:central_extension_and_braid_group}
The presentation of $\Aut_*(\cD_N)$ above allows us to write it as a central extension of $\PSL_{2}(\bZ)$:
    \[
    1\ra\bZ\ra\Aut_*(\cD_N)\xrightarrow{\alpha}\PSL_{2}(\bZ)\ra1.
    \]
The map $\bZ\ra\Aut_*(\cD_N)$ sends $1$ to the shift $[1]$ and $\alpha$ is given by
    \[
    T_1\mapsto\begin{pmatrix}1&1\\0&1\end{pmatrix},
    T_2\mapsto\begin{pmatrix}1&0\\-1&1\end{pmatrix}, \text{ and  }
    [1]\mapsto\begin{pmatrix}1&0\\0&1\end{pmatrix} \text{ in }\PSL_{2}(\bZ).
    \]

As a preliminary to the main result of this section, we need:
\begin{definition}
The group homomorphism $w\colon\Aut_*(\cD_N)\ra\bR$ is defined by setting
    \[
    w(T_1)=w(T_2)=\frac{4-3N}{6}\text{ \ and \ }w([1])=1.
    \]
\end{definition}

\begin{theorem}
\label{theorem:CYNA2quiver}
The map given by the shifting number
    \[
    \tau\colon\Aut_*(\cD_N)\ra\bR
    \]
is a quasimorphism. More precisely,
    \[
    \tau=w-\tfrac16(\phi\circ\alpha),
    \]
where $\phi$ is the (homogenization of the) Rademacher function on $\PSL_{2}(\bZ)$, see \autoref{app:quasihomomorphismPSL2}.
\end{theorem}

Recall the following shorthand notation used in \cite{FFHKL}.

\begin{definition}
Let $E,A_1,\ldots,A_n$ be objects in a triangulated category $\cD$. We write
    \[
    E\in\{A_1,\ldots,A_n\}
    \]
if there exists a sequence of exact triangles $B_{i-1}\ra B_i\ra A_i\xrightarrow{+1}$ for $1\leq i\leq n$, such that $B_0=0$ and $B_n=E$.
Note that the (ordered) elements $A_1,\ldots,A_n$ are not required to be the Harder--Narasimhan factors of $E$ with respect to a stability condition.
\end{definition}

We will make use of the following computations of $T_i(S_j)$ for $i,j\in\{1,2\}$, which can be established from the definition of the spherical twists, cf.~\cite[\S2]{BQS20}.

\begin{lemma}
\label{lem:A2sphericaltwist}
Let $N\geq3$ be an integer.
In $\cD_N=\cD_N(A_2)$, we have
\begin{align*}
    T_1(S_1)=S_1[1-N] & \text{ and }T_1(S_2)\in\{S_2,S_1\},\\
    T_2^{-1}(S_2)=S_2[N-1] & \text{ and } T_2^{-1}(S_1)\in\{S_2,S_1\}.
\end{align*}
\end{lemma}

The next result computes the upper and lower shifting numbers, assuming a specific presentation of an autoequivalence.
\begin{proposition}
    \label{proposition:hyperbolicshiftingnumber}
Take $F=T_1^{a_1}T_2^{-b_1}\cdots T_1^{a_k}T_2^{-b_k}\in\Aut_*(\cD_N)$ for some $a_i,b_i\geq0$.
Then
    \[
    \tau^+(F)=(N-1)\sum_{i=1}^kb_i \text{ \ and \ } \tau^-(F)=-(N-1)\sum_{i=1}^ka_i
    \]
Hence $\tau(F)=\frac{1-N}{2}(\sum_i(a_i-b_i))$.
\end{proposition}

\begin{proof}
Choose a split generator $G\coloneqq S_1\oplus S_2\in\cD_N$. Let $\sigma\in\Stab_*(\cD_N)$ be a stability condition such that $S_1$, $S_2$ and their shifts are the only indecomposable $\sigma$-semistable objects, and their phases satisfy $0<\phi_\sigma(S_1)<\phi_\sigma(S_2)<1$.
By \autoref{lem:A2sphericaltwist}, we have
    \[
    F^nG\in\{S_{i_1}^{\oplus m_1}[(N-1)n_1],\ldots,S_{i_k}^{\oplus m_k}[(N-1)n_k]\}
    \]
for some $i_1,\ldots,i_k\in\{1,2\}$, $m_1,\ldots,m_k>0$, and $n_1,\ldots,n_k\in\bZ$.
As an intermediate claim, we assert that
\begin{itemize}
    \item $\phi_\sigma^+(F^nG)=\max_{1\leq\ell\leq k}\{\phi_\sigma(S_{i_\ell})+(N-1)n_\ell\}$, and
    \item $\phi_\sigma^-(F^nG)=\min_{1\leq\ell\leq k}\{\phi_\sigma(S_{i_\ell})+(N-1)n_\ell\}$.
\end{itemize}

To establish the assertion, consider first the case when the phases of these objects are strictly decreasing, i.e.
    \[
    \phi_\sigma(S_{i_1})+(N-1)n_1>\cdots>\phi_\sigma(S_{i_k})+(N-1)n_k.
    \]
Then $\{S_{i_1}^{\oplus m_1}[(N-1)n_1],\ldots,S_{i_k}^{\oplus m_k}[(N-1)n_k]\}$ is the Harder--Narasimhan filtration of $F^nG$, and the assertion follows.

Next, if there are two consecutive terms with the same phase
    \[
    \phi_\sigma(S_{i_\ell})+(N-1)n_\ell=\phi_\sigma(S_{i_{\ell+1}})+(N-1)n_{\ell+1}
    \]
then $i_\ell=i_{\ell+1}$ and $n_\ell=n_{\ell+1}$.
Since $\Hom(S_i,S_i[1])=0$ for $i=1,2$, one can merge these two terms and obtain
\begin{multline*}
	F^nG\in\{S_{i_1}^{\oplus m_1}[(N-1)n_1],\ldots,S_{i_\ell}^{\oplus m_\ell+m_{\ell+1}}[(N-1)n_\ell], \\
	    S_{i_{\ell+2}}^{\oplus m_{\ell+2}}[(N-1)n_{\ell+2}],\ldots,S_{i_k}^{\oplus m_k}[(N-1)n_k]\}.
\end{multline*}
Therefore if the phases of these objects are nonincreasing
    \[
    \phi_\sigma(S_{i_1})+(N-1)n_1\geq\cdots\geq\phi_\sigma(S_{i_k})+(N-1)n_k,
    \]
then one can merge the terms with the same phases and obtain the Harder--Narasimhan filtration of $F^nG$, which proves the assertion.

Finally, suppose there exists some $\ell$ such that
    \[
    \phi_\sigma(S_{i_\ell})+(N-1)n_\ell<\phi_\sigma(S_{i_{\ell+1}})+(N-1)n_{\ell+1}.
    \]
Then the pair of objects $\{S_{i_\ell}^{\oplus m_\ell}[(N-1)n_\ell],S_{i_{\ell+1}}^{\oplus m_{\ell+1}}[(N-1)n_{\ell+1}]\}$ is of one of the following types:
\begin{itemize}
    \item $\{S_i^{\oplus m_\ell}[(N-1)n_\ell],S_i^{\oplus m_{\ell+1}}[(N-1)(n_{\ell}+m)]\}$ for some $i\in\{1,2\}$ and $m>0$,
    \item $\{S_2^{\oplus m_\ell}[(N-1)n_\ell],S_1^{\oplus m_{\ell+1}}[(N-1)(n_{\ell}+m)]\}$ for some $m>0$,
    \item $\{S_1^{\oplus m_\ell}[(N-1)n_\ell],S_2^{\oplus m_{\ell+1}}[(N-1)(n_{\ell}+m)]\}$ for some $m\geq0$.
\end{itemize}
Using the fact that $\Hom^\bullet(S_i,S_i)=\bC\oplus\bC[-N]$, $\Hom^\bullet(S_1,S_2)=\bC[-1]$, $\Hom^\bullet(S_2,S_1)=\bC[1-N]$, and the assumption that $N\geq3$, it follows that there are no nontrivial extensions, i.e.
    \[
    \Hom(S_{i_{\ell+1}}^{\oplus m_{\ell+1}}[(N-1)n_{\ell+1}],S_{i_\ell}^{\oplus m_\ell}[(N-1)n_\ell][1])=0,
    \]
between the relevant objects.
Therefore one can swap the order of $\{S_{i_\ell}^{\oplus m_\ell}[(N-1)n_\ell],S_{i_{\ell+1}}^{\oplus m_{\ell+1}}[(N-1)n_{\ell+1}]\}$ and still have
    \begin{multline*}
    F^nG\in\{S_{i_1}^{\oplus m_1}[(N-1)n_1],\ldots,S_{i_{\ell+1}}^{\oplus m_{\ell+1}}[(N-1)n_{\ell+1}],\\
    S_{i_\ell}^{\oplus m_\ell}[(N-1)n_\ell],\ldots,S_{i_k}^{\oplus m_k}[(N-1)n_k]\}.
    \end{multline*}
Hence one can reorder the sequence $\{S_{i_1}^{\oplus m_1}[(N-1)n_1],\ldots,S_{i_k}^{\oplus m_k}[(N-1)n_k]\}$ and assume that they are of decreasing phases. The claim then follows from the previous argument.

Now we can compute $\phi_\sigma^\pm(F^nG)$ where $G=S_1\oplus S_2$ and $F=T_1^{a_1}T_2^{-b_1}\cdots T_1^{a_k}T_2^{-b_k}\in\Aut_*(\cD_N)$ for some $a_i,b_i\geq0$.
Using \autoref{lem:A2sphericaltwist} we compute that
    \[
    \phi_\sigma^+(F^nG)=\max_{1\leq\ell\leq k}\{\phi_\sigma(S_{i_\ell})+(N-1)n_\ell\}=\phi_\sigma(S_2)+n(N-1)\sum_{i=1}^kb_i
    \]
and
    \[
    \phi_\sigma^-(F^nG)=\min_{1\leq\ell\leq k}\{\phi_\sigma(S_{i_\ell})+(N-1)n_\ell\}=\phi_\sigma(S_1)-n(N-1)\sum_{i=1}^ka_i.
    \]
By \autoref{theorem:computeviastability}, we have
    \[
    \tau^+(F)=\lim_{n\ra\infty}\Bigg(\frac{\phi_\sigma(S_2)}{n}+(N-1)\sum_{i=1}^kb_i\Bigg)=(N-1)\sum_{i=1}^kb_i
    \]
and 
    \[
    \tau^-(F)=\lim_{n\ra\infty}\Bigg(\frac{\phi_\sigma(S_1)}{n}-(N-1)\sum_{i=1}^ka_i\Bigg)=-(N-1)\sum_{i=1}^ka_i
    \]
as claimed.
\end{proof}

\subsubsection{Proof of \autoref{theorem:CYNA2quiver}}
    \label{sssec:proof_of_theorem:cyna2quiver}
Let $F\in \Aut_*(\cD_N)$ be an autoequivalence.
First consider the case when $\alpha(F)\in\PSL_{2}(\bZ)$ is of finite order.
Then $\alpha(F)$ is conjugate to either
    \[
    \alpha(T_1T_2T_1)=\begin{pmatrix}0&1\\-1&0\end{pmatrix},
    \alpha(T_2T_1)=\begin{pmatrix}1&1\\-1&0\end{pmatrix},
    \alpha((T_2T_1)^2)=\begin{pmatrix}0&1\\-1&-1\end{pmatrix},\text{ or }
    \mathbb{I}_2.
    \]
Hence $F$ can be written as
    \[
    F=g\widetilde{F}g^{-1}[n]
    \]
for some $\widetilde{F}\in\{T_1T_2T_1,T_2T_1,(T_2T_1)^2,\mathrm{id}_{\cD_N}\}$, $g\in\Aut_*(\cD_N)$, and $n\in\bZ$.
Observe that there exists a power $k\geq1$ such that $F^k=[\ell]$ for some $\ell\in\bZ$, since $(T_1T_2)^3=[4-3N]$.
Hence the shifting numbers satisfy $\tau(F)=\tau^\pm(F)=w(F)$ by \autoref{proposition:basicproperties}\ref{proposition:compatiblewithshifts}\ref{proposition:compatiblewithpowers}.
Note that $\phi(\alpha(F))=0$ in this case since the quasimorphism $\phi$ vanishes on finite order elements in $\PSL_{2}(\bZ)$.

Next consider the case when $\alpha(F)\in\PSL_{2}(\bZ)$ is of infinite order.
Then $\alpha(F)$ is conjugate to a positive word in
\[
\alpha(T_1)=\begin{pmatrix}1&1\\0&1\end{pmatrix}\ \text{ and }\ 
\alpha(T_2^{-1})=\begin{pmatrix}1&0\\1&1\end{pmatrix}.
\]
Hence $F$ can be written as
    \[
    F=gT_1^{a_1}T_2^{-b_1}\cdots T_1^{a_k}T_2^{-b_k}g^{-1}[n]
    \]
for some $a_1,b_1,\ldots,a_k,b_k\geq0$ (not all zero), $g\in\Aut_*(\cD)$, and $n\in\bZ$.
Using \autoref{proposition:hyperbolicshiftingnumber} and \autoref{proposition:basicproperties}\ref{proposition:compatiblewithshifts}\ref{proposition:conjugacyinvariant_tau} we find
    \[
    \tau(F)=\frac{1-N}{2}\Big(\sum_{i=1}^k(a_i-b_i)\Big)+n.
    \]
Recall that 
    \[
    w(F)=\frac{4-3N}{6}\Big(\sum_{i=1}^k(a_i-b_i)\Big)+n
    \]
and
    \[
    \phi(\alpha(F))=\sum_{i=1}^k(a_i-b_i).
    \]
This proves that the shifting number can be expressed as $\tau=w-\frac16(\phi\circ\alpha)$.
Since $w$ and $\alpha$ are group homomorphisms and $\phi$ is a quasimorphism, it follows that $\tau$ is a quasimorphism.
\hfill \qed

\begin{remark}
The maps $\tau^\pm\colon\Aut(\cD_N)\ra\bR$ given by the upper or lower shifting numbers are not quasimorphisms.
For instance,
    \[
    |\tau^+(T_1^n)+\tau^+(T_1^{-n})-\tau^+(\id)|=n(N-1)
    \]
for any $n>0$ by \autoref{thm:spherical_twist_calculation} and \autoref{proposition:basicproperties}\ref{proposition:compatiblewithpowers}\ref{proposition:inverse_tau}.
\end{remark}




\section{Quasimorphisms on Lie groups}
    \label{sec:quasimorphisms_on_lie_groups}

\paragraph{Outline}
We describe the central $\bZ$-extension of a Lie group of Hermitian type and its associated quasimorphism in \autoref{ssec:so2_rho}.
In \autoref{ssec:autoequivalences_and_lie_groups} we then apply this construction to the group of autoequivalences of an abelian surface and connect this quasimorphism to the shifting number constructed earlier in \autoref{theorem:abelsurface}.


\subsection{\texorpdfstring{$\SO(2,\rho)$}{SO(2,rho)}}
    \label{ssec:so2_rho}

We describe a classical quasimorphism on the universal cover of the orthogonal group of signature $(2,\rho)$.
A reference for the analogous case of $\Sp_{2g}$ is Barge--Ghys \cite{BargeGhys1992_Cocycles-dEuler-et-de-Maslov}, and for a reference covering all groups of hermitian type see Burger--Iozzi--Wienhard \cite[\S7]{BurgerIozziWienhard2010_Surface-group-representations-with-maximal-Toledo-invariant}.

\subsubsection{Setup}
    \label{sssec:setup_so2_rho}
Let $\cN_{\bR}$ be a real vector space equipped with a nondegenerate inner product of signature $(2,\rho)$, with $\rho\geq 1$.
The inner product of two vectors $v,w$ is denoted by $\ip{v,w}$.
The complexification is denoted by $\cN_{\bC}:=\cN_{\bR}\otimes_{\bR}\bC$.

Inside $\bP\left(\cN_{\bC}\right)$ we have the quadric of null-lines $Q(\cN_{\bC})$, i.e. $[v]$ such that $\ip{[v],[v]}=0$.
Inside $Q(\cN_{\bC})$ we have the hermitian symmetric space $\bD$ defined by
\[
    \bD:=\{[v] \in \bP(\cN_\bC)\quad \colon \quad \ip{[v],[v]}=0 \quad \ip{v,\ov{v}}>0\}
\]
Over $\bD$ we have a variation of weight $2$ polarized Hodge structure, of K3 type.
The Hodge decomposition $H^{2,0}\oplus H^{1,1}\oplus H^{0,2}$ over a point $[v]\in \bD$ is given by:
\[
    H^{2,0}=[v] \quad H^{0,2}:=[\ov{v}] \quad H^{1,1}:=\left(H^{2,0}\oplus H^{0,2}\right)^{\perp}
\]
where the last orthogonal complement is for the indefinite inner product (and it is also an orthogonal complement for the positive-definite inner product induced by the Hodge structure).
Define also
\begin{align}
    \label{eqn:cP_definition_homogeneous}
    \cP:=\left(H^{2,0}\oplus H^{0,2}\right)\setminus \{v\colon \ip{v,\ov{v}}=0\}
\end{align}
We view $\cP$ naturally as a subset of $\cN_{\bC}$ and observe that it has two connected components.
Denote by $\cP^+$ the one that contains $H^{2,0}$.
The most direct way to see the structure of $\cP$ is to quotient by the free $\bC^\times$-action and obtain a bundle over $\bD$, with fibers $\bP\left(H^{2,0}\oplus H^{0,2}\right)\isom \bP^1(\bC)$ with equators, corresponding to real vectors, removed.

In particular, observe that the inclusion $H^{2,0}_\times\into \cP^+$ induces a homotopy equivalence, where $H^{2,0}_{\times}$ denotes the bundle with the zero section removed.

\begin{remark}[On the $\GL_2(\bR)$-action]
    \label{rmk:on_the_gl_2}
    The group $\GL_2(\bR)$ naturally acts on $\cN_{\bC}:=\cN_\bR\otimes_\bR \bC\isom \cN_{\bR}\otimes \bR^2$ via its action on $\bR^2$.
    The subset $\cP\subset \cN_{\bC}$ is invariant under this action and the quotient is identified with $\bD$.
    If we restrict to $\GL_2^+(\bR)$ then the quotient is two copies of $\bD$, corresponding to the two components of $\cP$.

    Note also that the map $\cP^+\to \bD$ is not holomorphic, even though restricted to the subset $H^{2,0}_\times\subset \cP^+$ it is.
\end{remark}

\subsubsection{The associated groups}
    \label{sssec:the_associated_groups}
Denote by $G:=\SO^\circ(\cN_{\bR})$ the connected component of the identity of the group of real isometries of the indefinite pairing on $\cN_{\bR}$.
Then $G$ acts transitively on $\bD$, with stabilizer of a point $[v_0]$ equal to a maximal compact subgroup $K$ of $G$.
Furthermore, the variation of Hodge structure and other spaces above also admit $G$-actions and the maps between spaces are $G$-equivariant.

The group $K$ is isomorphic to $\SO_2(\bR)\times \SO_{\rho}(\bR)$, where the $\SO_2(\bR)$ factor corresponds to real isometries of $H^{2,0}\oplus H^{0,2}$ and $\SO_{\rho}(\bR)$ corresponds to real isometries of $H^{1,1}$.
Furthermore, the $\SO_2(\bR)$-factor is naturally identified with the unitary rotations $\bbU(1)$ of $H^{2,0}$ (over the basepoint $[v_0]$).

\subsubsection{$\bZ$-covers of groups and spaces}
    \label{sssec:bz_covers_groups_spaces}
Observe that $K$ is not simply connected and we denote by $\wtilde{K}$ the $\bZ$-cover of $K$ corresponding to the cover of $\SO_2(\bR)$ by $\bR$.
We ignore the additional cover that might come from the case $\rho=2$.
Recall also that $G$ is homotopy-equivalent to $K$, so we also have a $\bZ$-cover $\wtilde{G}\to G$.
Both maps give central $\bZ$-extensions of the groups $G$ and $K$ respectively.

Similarly, observe that the space $\cP^+$ is also not simply connected.
Indeed if we take the quotient of $\cP^+$ by the free $\bC^\times$-action, we get a disk bundle over $\bD$.
The total space of the disk bundle is simply connected, even contractible, so $\cP^+$ has fundamental group $\bZ$.
Recall also that the disk bundle has a reference point, given by $H^{2,0}$.
Fix now a lift $v_0\in H^{2,0}([v_0])=[v_0]$ of the basepoint $[v_0]\in \bD$ and define the associated universal covers $\wtilde{\cP}^+\to \cP$ and $\wtilde{H}^{2,0}_\times \to H^{2,0}_{\times}$.

We have natural actions of the groups: of $G$ on $\cP^+$ preserving $H^{2,0}_\times$, and of $\wtilde{G}$ on $\wtilde{\cP}^+$ preserving $\wtilde{H}^{2,0}_\times$.

\subsubsection{Quasimorphism on $\wtilde{G}$}
    \label{sssec:quasimorphism_on_wtildeg}
We follow \cite[Part C, Prop.~1.2]{BargeGhys1992_Cocycles-dEuler-et-de-Maslov} in the setting of our group $G$ isomorphic to $\SO_{2,\rho}(\bR)$, with some reinterpretations.
Note that the construction of the quasimorphism involves only the bundle $H^{2,0}_\times$ and its universal cover.

Fix a real isotropic vector, i.e.~$w_0\in \cN_\bR\backslash\{0\}$ such that $\ip{w_0,w_0}=0$.
It follows that for any $v\in \cN_{\bC}$ such that $[v]\in \bD$ we have that $\ip{v,w_0}\neq 0$, since relative to the Hodge decomposition determined by $[v]$, the vector $w_0$ must have nontrivial $(2,0)$ and $(0,2)$-components.
Define for $g\in G$ the function
\[
    j_{[v]}(g):= \frac{\ip{(gv),w_0}}{\ip{v,w_0}}\in \bC^\times \text{ which only depends on }[v],\text{ not }v.
\]
We have the basic calculation
\begin{align}
    \label{eqn:cocycle_j}
    \begin{split}
    j_{[v]}(g_1\cdot g_2) & = 
    \frac{\ip{g_1g_2v,w_0}}{\ip{g_2v,w_0}} \cdot 
    \frac{\ip{g_2v,w_0}}{\ip{v.w_0}}\\
    & = j_{g_2[v]}(g_1) \cdot j_{[v]}(g_2)
    \end{split}
\end{align}
Define next the function $\phi_{[v]}(g):=\arg j_{[v]}(g)\in \bR/\bZ$ and its lift to the universal cover $\wtilde{G}$ of $G$:
\[
    \Phi_{\wtilde{[v]}}(\wtilde{g})=\wtilde{\arg} \wtilde{j}_{\wtilde{[v]}}(\wtilde{g})
\]
which is the unique continuous extension to $\bR$ of the pullback $\phi_{[v]}(g)$ from $g$, viewed as a map to $\bR/\bZ$.

Note for future reference that $\Phi_{\wtilde{[v]}}(\wtilde{g}+1)=\Phi_{\wtilde{[v]}}(\wtilde{g})+1$ where $1$ in the first expression denote adding $1$ in the center of $\wtilde{G}$.
Similarly $\Phi_{l\cdot \wtilde{[v]}}(\wtilde{g}) = \Phi_{\wtilde{[v]}}(\wtilde{g})$ where $l$ is in the center of $\wtilde{G}$ again.

\begin{theorem}[Quasimorphism on orthogonal group]
    \label{thm:quasimorphism_so2rho}
    For a fixed $[\wtilde{v_0}]$ the map $\Phi_{[\wtilde{v_0}]}\colon \wtilde{G}\to \bR$ is a quasimorphism.
\end{theorem}
\begin{proof}
    The basic cocycle relation in \autoref{eqn:cocycle_j} generalizes straightforwardly to $\Phi$ to give
    \[
        \Phi_{[\wtilde{v_0}]}
        (\wtilde{g_1}\wtilde{g_2}) =
        \Phi_{\wtilde{g_2} [\wtilde{v_0}]}(\wtilde{g_1})
        + \Phi_{[\wtilde{v_0}]}(\wtilde{g_2})
    \]
    so it suffices to check that
    \begin{align}
        \label{eqn:defect_to_check}
        \abs{
        \Phi_{\wtilde{g_2}[\wtilde{v_0}]}(\wtilde{g_1})
        -
        \Phi_{[\wtilde{v_0}]}(\wtilde{g_1})
        }
        \leq C \quad \text{ for }C\text{ independent of }\wtilde{g_1},\wtilde{g_2}.
    \end{align}

    Observe next that adding elements of the center of $\wtilde{G}$ to $\wtilde{g_1}$ does not change the expression, so above we can replace $\wtilde{g_1}$ by its projection $g_1\in G$.
    Similarly shifting the basepoint $\wtilde{[v_0]}$ by an element of the center does not affect the expression, and analogously for $\wtilde{g_2}$.
    So the expression to be estimated depends only on the projections of the parameters to $\bD$ and $G$.

    With these simplifications, the quantity in \autoref{eqn:defect_to_check} is bounded, up to within $1$, by the following geometric number.
    Connect $[v_0]$ and $g_2[v_0]$ by a path $\gamma_t$ staying in $\bD$, and count (with sign) the total number of times $\arg j_{\gamma_t}(g_1)$ crosses the origin.

    In order to perform this computation, select representatives $v_0$ and $v_1$ for $[v_0]$ and $g_2[v_0]$ such that $v_0.w_0=1=v_1.w_0$.
    The function $j_{[v]}(g)$ is independent of the choice of representative $v$, but when $v.w_0=1$ the function is polynomial in the entries of $g$.
    Taking $\gamma_t=t v_0 + (1-t)v_1$ makes the function $j_{\gamma_t}(g_1)$ a polynomial in the entries of $g_1$, of degree bounded by $\rho+2$.
    It follows that this function crosses the real axis at most $\rho+2$ times, so the defect in \autoref{eqn:defect_to_check} is at most $\rho+3$.
\end{proof}

\begin{remark}[Sections of the canonical bundle]
    \label{rmk:sections_of_the_canonical_bundle}
    The reader will recognize that the above constructions use implicitly a trivialization of the canonical bundle of the hermitian domain $\bD$, using the vector $w_0$ fixed initially.
    Incidentally, the same construction that gives a nonvanishing holomorphic section of $H^{1,0}$ (for weight $1$ variations of Hodge structure) and $H^{2,0}$ (for weight $2$ variations of Hodge structure of K3 type) is used in \cite{Kontsevich_Lyapunov-exponents-and-Hodge-theory} and \cite{Filip_Families-of-K3-surfaces-and-Lyapunov-exponents} respectively to compute Lyapunov exponents of the corresponding local systems.
\end{remark}



\subsection{Autoequivalences and Lie groups}
    \label{ssec:autoequivalences_and_lie_groups}

\subsubsection{Setup}
    \label{sssec:setup_autoequivalences_and_lie_groups}

\begin{theorem}[Quasimorphisms for abelian surfaces]
    \label{thm:quasimorphisms_for_abelian_surfaces}
    Suppose that $X$ is an abelian surface of Picard rank $\rho$, $\cD$ is its derived category, and $\cN_\bZ$ is its Mukai lattice, of signature $(2,\rho)$.
    \leavevmode
    \begin{enumerate}
        \item Let $\Aut^{\circ}\cD$ denote the finite index subgroup of $\Aut \cD$ that maps to $\SO^\circ\left(\cN_{\bR}\right)$ under the natural map $\Aut \cD\to \Orthog(\cN_{\bR})$.
        Then there exists a lift
        \[
            \Aut^\circ \cD \to \wtilde{\SO^\circ}(\cN_\bZ)\into \wtilde{\SO^\circ}(\cN_\bR)
        \]
        to the central extension of $\SO^\circ(\cN_{\bZ})$ coming from the central extension of the corresponding Lie group constructed in \autoref{sssec:bz_covers_groups_spaces}.
        The double shift functor $[2]$ maps to the generator of the center.
        \item The homogenization of the quasimorphism in \autoref{thm:quasimorphism_so2rho} agrees with twice the shifting number quasimorphism from \autoref{theorem:abelsurface}, when restricted to $\Aut^\circ \cD$.
    \end{enumerate}
\end{theorem}
\begin{proof}
    The existence of the lift and its properties follow from \cite[Thm.~15.2]{Bridgeland2008_Stability-conditions-on-K3-surfaces}.

    The agreement of the homogeneous quasimorphisms, one coming from the Lie group and the other from the shifting number, follows from the agreement of their pre-homogenized versions $\wtilde{\tau}$ in \autoref{theorem:abelsurface} and $\Phi_{[\wtilde{v_0}]}$ in \autoref{thm:quasimorphism_so2rho}.
    Indeed $\wtilde{\tau}(F)=\phi_{\sigma_0}(F(k(x)))$, so we can take $[\wtilde{v_0}]=\sigma_0$ under Bridgeland's identification of $\Stab^\dagger(X)$ and $\wtilde{\cP^+}$ (loc. cit.).
    We also take $w_0$ in the construction of $\Phi_{[\wtilde{v_0}]}$ to be the Mukai vector of $k(x)$.
\end{proof}

\begin{remark}[The case of K3 surfaces]
    \label{rmk:the_case_of_k3_surfaces}
    In the case of K3 surfaces, we cannot expect an agreement of the two quasimorphisms.
    According to \autoref{thm:spherical_twist_calculation} spherical twists have non-trivial shifting numbers when $N=2$, but their second powers generate a group that's (at least conjecturally for $\rho\geq 2$) ``disjoint'' from the relevant central extension, see \cite[Thm.~1,4]{BayerBridgeland2017_Derived-automorphism-groups-of-K3-surfaces-of-Picard-rank-1}.

    Nonetheless, let us note that the construction of \autoref{thm:quasimorphisms_for_abelian_surfaces}(i) works for K3 surfaces and yields a quasimorphism there.
    We do not know if the shifting number $\tau$ also gives a quasimorphism on $\Aut \cD$ in this case.
\end{remark}





\appendix


\section{An explicit quasimorphism on \texorpdfstring{$\PSL_{2}(\bZ)$}{PSL(2,Z)}}
    \label{app:quasihomomorphismPSL2}

\paragraph{Outline of section}
In this section, we define an explicit quasimorphism $\phi\colon\PSL_{2}(\bZ)\ra\bR$ that is used in \autoref{sec:CYA2}.
More information and further details appear in \cite{BargeGhys1992_Cocycles-dEuler-et-de-Maslov}.

We start by recalling a fact about constructing quasimorphisms on free products of groups.
For a group $G$ define $C^\odd_b(G,\bR)$ to be the set of bounded functions $f\colon G\ra\bR$ such that $f(g)=-f(g^{-1})$ for any $g\in G$.

Let $\{G_s\}_{s\in S}$ be a collection of groups and let $G=\star_{s\in S}G_s$ be the associated free product group.
Then any element $x\in G\setminus 1$ can be uniquely written as $x=x_1x_2\cdots x_n$ such that $x_i\in G_{s_i}$ is nontrivial and $s_i\neq s_{i+1}$ for each $i$.

\begin{proposition}[{\cite[Prop.~4.1]{Rolli}}]
\label{proposition:quasihomomorphismfreeproduct}
Let $\{f_s\}_{s\in S}$ be a collection of functions such that $f_s\in C_b^\odd(G_s,\bR)$ and $\sup_{s\in S}\norm{f_s}_\infty<\infty$.
Then the function $g\colon G\ra\bR$ defined by
    \[
    g(x)=\sum_{i=1}^nf_{s_i}(x_i)
    \]
is a quasimorphism.
\end{proposition}

\subsubsection{Presentation and elements of \texorpdfstring{$\PSL_2(\bZ)$}{PSL2}}
    \label{sssec:presentation_and_elements_of_psl_2}
Recall that we have the following presentation:
\begin{align*}
    \PSL_{2}(\bZ)&=\left(\rightquot{\bZ}{2\bZ}\right)
    *
    \left(\rightquot{\bZ}{3\bZ}\right)
    =\ip{S,U\colon S^2=U^3=1}\\
    \intertext{ with matrices}
    S &=\begin{pmatrix}0&-1\\1&0\end{pmatrix}
    \quad
    U=\begin{pmatrix}0&-1\\1&1\end{pmatrix}
    \\
    L&=SU=\begin{pmatrix}1&1\\0&1\end{pmatrix}
    \quad
    R=SU^{-1}=\begin{pmatrix}1&0\\1&1\end{pmatrix}
\end{align*}

Any element in $\PSL_2\left(\bZ\right)$ can be uniquely written as
\begin{align}
    \label{eqn:SU_word}
    S^{\delta_1}U^{\ep_1}SU^{\ep_2}S\cdots SU^{\ep_m}S^{\delta_2}
\end{align}
for some $\delta_1,\delta_2\in\{0,1\}$ and $\ep_1,\ldots,\ep_m\in\{-1,1\}$.

\begin{definition}
Define $\phi_0\colon\PSL_{2}(\bZ)\ra\bR$ by
    \[
    \phi_0(A)\coloneqq\sum_i\ep_i
    \]
if $A=S^{\delta_1}U^{\ep_1}SU^{\ep_2}S\cdots SU^{\ep_m}S^{\delta_2}$ for some $\delta_1,\delta_2\in\{0,1\}$ and $\ep_1,\ldots,\ep_m\in\{-1,1\}$.
This is also called the \emph{Rademacher function} in \cite[\S B-4]{BargeGhys1992_Cocycles-dEuler-et-de-Maslov}

\end{definition}

This defines a quasimorphism on $\PSL_{2}(\bZ)$ by \autoref{proposition:quasihomomorphismfreeproduct}.

\begin{definition}
Define $\phi\colon\PSL_{2}(\bZ)\ra\bR$ to be the homogenization of $\phi_0$, i.e.
    \[
    \phi(A)\coloneqq\lim_{n\ra\infty}\frac{\phi_0(A^n)}{n}.
    \]
\end{definition}

It is a standard fact that the homogenization of a quasimorphism is again a quasimorphism.
Moreover, it is homogeneous in the sense that $\phi(A^n)=n\phi(A)$, and is constant on conjugacy classes.
There is an alternative description of $\phi$ which we use in \autoref{sec:CYA2}.

\begin{lemma}
Let $A$ be an element in $\PSL_{2}(\bZ)$.
    \begin{enumerate}
        \item If $A$ is of finite order, then $\phi(A)=0$.
        \item If $A$ is of infinite order, then it is conjugate to a positive word in $L$ and $R$, i.e.~$A=BL^{a_1}R^{b_1}\cdots L^{a_k}R^{b_k}B^{-1}$ for some $a_1,b_1,\ldots,a_k,b_k\geq0$ not all zero, and $B\in\PSL_{2}(\bZ)$.
        Moreover, we have
            \[
            \phi(A)=\sum_{i=1}^k(a_i-b_i).
            \]
    \end{enumerate}
\end{lemma}

\begin{proof}
Part (i) follows from the fact that $\phi$ is homogeneous.
Now suppose $A\in\PSL_{2}(\bZ)$ is of infinite order.
Let us see that $A$ can be conjugated to a positive word in $L=SU$ and $R=SU^{-1}$.
Using the presentation in \autoref{eqn:SU_word}, observe that by one conjugation by $S$ we can ensure that the last letter is not $S$.
If the first letter is $S$ then we are done, otherwise we conjugate by $U^{-\epsilon_1}$.
If $\epsilon_m + \epsilon_1\neq 0$ then we are done, otherwise we apply an $S$-conjugation again to remove the $S$ at the end and repeat this argument.
Either the process of conjugation stops, and the claim follows, or we are left at the end with $S$ or $U$, which are finite order (a contradiction since $A$ was assumed of infinite order and conjugacy does not change this fact).

Observe from the definition of $\phi_0$ that if $W$ is a positive word in $L$ and $R$, then $\phi_0(W^n)=n\phi(W)$ for any $n\in\bN$, hence $\phi(W)=\phi_0(W)$. Since $\phi$ is constant on conjugacy classes, we have
    \begin{align*}
        \phi(A) & =\phi(L^{a_1}R^{b_1}\cdots L^{a_k}R^{b_k})\\
        & =\phi_0(L^{a_1}R^{b_1}\cdots L^{a_k}R^{b_k})\\
        & =\sum_{i=1}^k(a_i-b_i).
    \end{align*}
which is the required claim.
\end{proof}



\bibliographystyle{sfilip_bibstyle}
\bibliography{shifting_numbers}

\newcommand{\etalchar}[1]{$^{#1}$}
\providecommand{\bysame}{\leavevmode ---\ }
\providecommand{\andchar}{\&}
\begin{thebibliography}{DHKK14}

\bibitem[AM16]{ArcaraMiles}
\textsc{Arcara~D.{ \andchar\hskip 0.5em}Miles~E.} --- {``Bridgeland stability
  of line bundles on surfaces''}. {\em J. Pure Appl. Algebra} {\bfseries 220}
  no.~4, (2016) 1655--1677. \url{https://doi.org/10.1016/j.jpaa.2015.09.021}.

\bibitem[Bav91]{Bavard}
\textsc{Bavard~C.} --- {``Longueur stable des commutateurs''}. {\em Enseign.
  Math. (2)} {\bfseries 37} no.~1-2, (1991) 109--150.

\bibitem[BB17]{BayerBridgeland2017_Derived-automorphism-groups-of-K3-surfaces-of-Picard-rank-1}
\textsc{Bayer~A.{ \andchar\hskip 0.5em}Bridgeland~T.} --- {``Derived
  automorphism groups of {K}3 surfaces of {P}icard rank 1''}. {\em Duke Math.
  J.} {\bfseries 166} no.~1, (2017) 75--124.
  \url{https://doi.org/10.1215/00127094-3674332}.

\bibitem[BF02]{BestvinaFujiwara2002_Bounded-cohomology-of-subgroups-of-mapping-class-groups}
\textsc{Bestvina~M.{ \andchar\hskip 0.5em}Fujiwara~K.} --- {``Bounded
  cohomology of subgroups of mapping class groups''}. {\em Geom. Topol.}
  {\bfseries 6} (2002) 69--89. \url{https://doi.org/10.2140/gt.2002.6.69}.

\bibitem[BG92]{BargeGhys1992_Cocycles-dEuler-et-de-Maslov}
\textsc{Barge~J.{ \andchar\hskip 0.5em}Ghys~E.} --- {``Cocycles d'{E}uler et de
  {M}aslov''}. {\em Math. Ann.} {\bfseries 294} no.~2, (1992) 235--265.
  \url{https://doi.org/10.1007/BF01934324}.

\bibitem[BIW10]{BurgerIozziWienhard2010_Surface-group-representations-with-maximal-Toledo-invariant}
\textsc{Burger~M., Iozzi~A.,{ \andchar\hskip 0.5em}Wienhard~A.} --- {``Surface
  group representations with maximal {T}oledo invariant''}. {\em Ann. of Math.
  (2)} {\bfseries 172} no.~1, (2010) 517--566.
  \url{https://doi.org/10.4007/annals.2010.172.517}.

\bibitem[BM01]{BridgelandMaciocia}
\textsc{Bridgeland~T.{ \andchar\hskip 0.5em}Maciocia~A.} --- {``Complex
  surfaces with equivalent derived categories''}. {\em Math. Z.} {\bfseries
  236} no.~4, (2001) 677--697. \url{https://doi.org/10.1007/PL00004847}.

\bibitem[BO01]{BondalOrlov2001_Reconstruction-of-a-variety-from-the-derived-category-and-groups-of-autoequivalences}
\textsc{Bondal~A.{ \andchar\hskip 0.5em}Orlov~D.} --- {``Reconstruction of a
  variety from the derived category and groups of autoequivalences''}. {\em
  Compositio Math.} {\bfseries 125} no.~3, (2001) 327--344.
  \url{https://doi.org/10.1023/A:1002470302976}.

\bibitem[BQS20]{BQS20}
\textsc{Bridgeland~T., Qiu~Y.,{ \andchar\hskip 0.5em}Sutherland~T.} ---
  {``Stability conditions and the {A}2 quiver''}. {\em Adv. Math.} {\bfseries
  365} (2020) 107049. \url{https://doi.org/10.1016/j.aim.2020.107049}.

\bibitem[Bri07]{Bridgeland2007_Stability-conditions-on-triangulated-categories}
\textsc{Bridgeland~T.} --- {``Stability conditions on triangulated
  categories''}. {\em Ann. of Math. (2)} {\bfseries 166} no.~2, (2007)
  317--345. \url{https://doi.org/10.4007/annals.2007.166.317}.

\bibitem[Bri08]{Bridgeland2008_Stability-conditions-on-K3-surfaces}
\bysame , {``Stability conditions on {$K3$} surfaces''}. {\em Duke Math. J.}
  {\bfseries 141} no.~2, (2008) 241--291.
  \url{https://doi.org/10.1215/S0012-7094-08-14122-5}.

\bibitem[Cal09]{Calegari2009_scl}
\textsc{Calegari~D.} --- {\em scl}, vol.~20 of {\em MSJ Memoirs}. Mathematical
  Society of Japan, Tokyo --- 2009. \url{https://doi.org/10.1142/e018}.

\bibitem[DHKK14]{DimitrovHaidenKatzarkov_Dynamical-systems-and-categories}
\textsc{Dimitrov~G., Haiden~F., Katzarkov~L.,{ \andchar\hskip
  0.5em}Kontsevich~M.} --- {``Dynamical systems and categories''}. in {\em The
  influence of {S}olomon {L}efschetz in geometry and topology} --- vol.~621 of
  {\em Contemp. Math.}, pp.~133--170. Amer. Math. Soc., Providence, RI ---
  2014. \url{https://doi.org/10.1090/conm/621/12421}.

\bibitem[EL19]{ElaginLunts}
\textsc{Elagin~A.{ \andchar\hskip 0.5em}Lunts~V.~A.} --- {``{Three notions of
  dimension for triangulated categories}''}. {\em arXiv e-prints} (Jan., 2019)
  arXiv:1901.09461.

\bibitem[EP00]{EliashbergPolterovich2000_Partially-ordered-groups-and-geometry-of-contact-transformations}
\textsc{Eliashberg~Y.{ \andchar\hskip 0.5em}Polterovich~L.} --- {``Partially
  ordered groups and geometry of contact transformations''}. {\em Geom. Funct.
  Anal.} {\bfseries 10} no.~6, (2000) 1448--1476.
  \url{https://doi.org/10.1007/PL00001656}.

\bibitem[EP03]{EntovPolterovich2003_Calabi-quasimorphism-and-quantum-homology}
\textsc{Entov~M.{ \andchar\hskip 0.5em}Polterovich~L.} --- {``Calabi
  quasimorphism and quantum homology''}. {\em Int. Math. Res. Not.} no.~30,
  (2003) 1635--1676. \url{https://doi.org/10.1155/S1073792803210011}.

\bibitem[Fan18a]{Fan2018_Entropy-of-an-autoequivalence-on-Calabi-Yau-manifolds}
\textsc{Fan~Y.-W.} --- {``Entropy of an autoequivalence on {C}alabi-{Y}au
  manifolds''}. {\em Math. Res. Lett.} {\bfseries 25} no.~2, (2018) 509--519.
  \url{https://doi.org/10.4310/MRL.2018.v25.n2.a8}.

\bibitem[Fan18b]{2018arXiv180110485F}
\bysame , {``{On entropy of P-twists}''}. {\em arXiv e-prints} (Jan., 2018)
  arXiv:1801.10485, \href{http://arxiv.org/abs/1801.10485}{{\ttfamily
  arXiv:1801.10485 [math.AG]}}.

\bibitem[FFH{\etalchar{+}}19]{FFHKL}
\textsc{{Fan}~Y.-W., {Filip}~S., {Haiden}~F., {Katzarkov}~L.,{ \andchar\hskip
  0.5em}{Liu}~Y.} --- {``{On pseudo-Anosov autoequivalences}''}. {\em arXiv
  e-prints} (Oct., 2019) arXiv:1910.12350.

\bibitem[FFO20]{FFO}
\textsc{{Fan}~Y.-W., {Fu}~L.,{ \andchar\hskip 0.5em}{Ouchi}~G.} ---
  {``{Categorical polynomial entropy}''}. {\em arXiv e-prints} (Mar., 2020)
  arXiv:2003.14224.

\bibitem[Fil18]{Filip_Families-of-K3-surfaces-and-Lyapunov-exponents}
\textsc{Filip~S.} --- {``Families of {K}3 surfaces and {L}yapunov exponents''}.
  {\em Israel J. Math.} {\bfseries 226} no.~1, (2018) 29--69.
  \url{https://doi.org/10.1007/s11856-018-1682-4}.

\bibitem[Ghy01]{Ghys2001_Groups-acting-on-the-circle}
\textsc{Ghys~E.} --- {``Groups acting on the circle''}. {\em Enseign. Math.
  (2)} {\bfseries 47} no.~3-4, (2001) 329--407.

\bibitem[{Gin}06]{Ginzburg2006_Calabi-Yau-algebras}
\textsc{{Ginzburg}~V.} --- {``{Calabi-Yau algebras}''}. {\em arXiv Mathematics
  e-prints} (Dec., 2006) math/0612139.

\bibitem[Giv90]{Givental1990_The-nonlinear-Maslov-index}
\textsc{Givental~A.~B.} --- {``The nonlinear {M}aslov index''}. in {\em
  Geometry of low-dimensional manifolds, 2 ({D}urham, 1989)} --- vol.~151 of
  {\em London Math. Soc. Lecture Note Ser.}, pp.~35--43. Cambridge Univ. Press,
  Cambridge --- 1990.

\bibitem[HT06]{HuybrechtsThomas2006_Bbb-P-objects-and-autoequivalences-of-derived-categories}
\textsc{Huybrechts~D.{ \andchar\hskip 0.5em}Thomas~R.} ---
  {``{$\mathbb{P}$}-objects and autoequivalences of derived categories''}. {\em
  Math. Res. Lett.} {\bfseries 13} no.~1, (2006) 87--98.
  \url{https://doi.org/10.4310/MRL.2006.v13.n1.a7}.

\bibitem[{Ike}16]{Ikeda_Mass-growth-of-objects-and-categorical-entropy}
\textsc{{Ikeda}~A.} --- {``{Mass growth of objects and categorical entropy}''}.
  {\em arXiv e-prints} (Dec, 2016) arXiv:1612.00995.

\bibitem[Kel11]{Keller}
\textsc{Keller~B.} --- {``Deformed {C}alabi-{Y}au completions''}. {\em J. Reine
  Angew. Math.} {\bfseries 654} (2011) 125--180.
  \url{https://doi.org/10.1515/CRELLE.2011.031}. With an appendix by Michel Van
  den Bergh.

\bibitem[Kik17]{Kikuta}
\textsc{Kikuta~K.} --- {``On entropy for autoequivalences of the derived
  category of curves''}. {\em Advances in Mathematics} {\bfseries 308} (Feb,
  2017) 699--712. \url{http://dx.doi.org/10.1016/j.aim.2016.12.027}.

\bibitem[{Kik}19]{KikutaCurvature}
\textsc{{Kikuta}~K.} --- {``{Curvature of the space of stability
  conditions}''}. {\em arXiv e-prints} (July, 2019) arXiv:1907.10973 ---
  \href{http://arxiv.org/abs/1907.10973}{{\ttfamily arXiv:1907.10973
  [math.AG]}}.

\bibitem[Kon97]{Kontsevich_Lyapunov-exponents-and-Hodge-theory}
\textsc{Kontsevich~M.} --- {``Lyapunov exponents and {H}odge theory''}. in {\em
  The mathematical beauty of physics ({S}aclay, 1996)} --- vol.~24 of {\em Adv.
  Ser. Math. Phys.}, pp.~318--332. World Sci. Publ., River Edge, NJ --- 1997.

\bibitem[Kot04]{Kotschick2004_Quasi-homomorphisms-and-stable-lengths-in-mapping-class-groups}
\textsc{Kotschick~D.} --- {``Quasi-homomorphisms and stable lengths in mapping
  class groups''}. {\em Proc. Amer. Math. Soc.} {\bfseries 132} no.~11, (2004)
  3167--3175. \url{https://doi.org/10.1090/S0002-9939-04-07508-2}.

\bibitem[KOT19]{KikutaOuchiTakahashi}
\textsc{{Kikuta}~K., {Ouchi}~G.,{ \andchar\hskip 0.5em}{Takahashi}~A.} ---
  {``{Serre dimension and stability conditions}''}. {\em arXiv e-prints} (July,
  2019) arXiv:1907.10981 --- \href{http://arxiv.org/abs/1907.10981}{{\ttfamily
  arXiv:1907.10981 [math.AG]}}.

\bibitem[KS08]{KontsevichSoibelman08}
\textsc{{Kontsevich}~M.{ \andchar\hskip 0.5em}{Soibelman}~Y.} --- {``{Stability
  structures, motivic Donaldson-Thomas invariants and cluster
  transformations}''}. {\em arXiv e-prints} (Nov., 2008) arXiv:0811.2435 ---
  \href{http://arxiv.org/abs/0811.2435}{{\ttfamily arXiv:0811.2435 [math.AG]}}.

\bibitem[KST20]{KST20}
\textsc{Kikuta~K., Shiraishi~Y.,{ \andchar\hskip 0.5em}Takahashi~A.} --- {``A
  note on entropy of auto-equivalences: lower bound and the case of orbifold
  projective lines''}. {\em Nagoya Math. J.} {\bfseries 238} (2020) 86--103.
  \url{https://doi.org/10.1017/nmj.2018.21}.

\bibitem[KT19]{KiTa}
\textsc{Kikuta~K.{ \andchar\hskip 0.5em}Takahashi~A.} --- {``On the categorical
  entropy and the topological entropy''}. {\em Int. Math. Res. Not. IMRN}
  no.~2, (2019) 457--469. \url{https://doi.org/10.1093/imrn/rnx131}.

\bibitem[Ouc20]{Ouchi2020_On-entropy-of-spherical-twists}
\textsc{Ouchi~G.} --- {``On entropy of spherical twists''}. {\em Proc. Amer.
  Math. Soc.} {\bfseries 148} no.~3, (2020) 1003--1014.
  \url{https://doi.org/10.1090/proc/14762}.

\bibitem[Poi85]{Poincare1885_Sur-les-courbes-definies-par-les-equations-differentielles-III}
\textsc{Poincar\'{e}~H.} --- {``Sur les courbes d\'{e}finies par les
  \'{e}quations diff\'{e}rentielles (iii)''}. {\em Journal de Math\'{e}matiques
  Pures et Appliqu\'{e}es} {\bfseries 1} (1885) 167--244.
  \url{http://eudml.org/doc/235596}.

\bibitem[{Rol}09]{Rolli}
\textsc{{Rolli}~P.} --- {``{Quasi-morphisms on Free Groups}''}. {\em arXiv
  e-prints} (Nov., 2009) arXiv:0911.4234 ---
  \href{http://arxiv.org/abs/0911.4234}{{\ttfamily arXiv:0911.4234 [math.GR]}}.

\bibitem[Rue85]{Ruelle1985_Rotation-numbers-for-diffeomorphisms-and-flows}
\textsc{Ruelle~D.} --- {``Rotation numbers for diffeomorphisms and flows''}.
  {\em Ann. Inst. H. Poincar\'{e} Phys. Th\'{e}or.} {\bfseries 42} no.~1,
  (1985) 109--115. \url{http://www.numdam.org/item?id=AIHPB_1985__42_1_109_0}.

\bibitem[Smi18]{Smith2018_Stability-conditions-in-symplectic-topology}
\textsc{Smith~I.} --- {``Stability conditions in symplectic topology''}. in
  {\em Proceedings of the {I}nternational {C}ongress of
  {M}athematicians---{R}io de {J}aneiro 2018. {V}ol. {II}. {I}nvited lectures}
  --- pp.~969--991. World Sci. Publ., Hackensack, NJ --- 2018.

\bibitem[ST01]{SeidelThomas_Braid-group-actions-on-derived-categories-of-coherent}
\textsc{Seidel~P.{ \andchar\hskip 0.5em}Thomas~R.} --- {``Braid group actions
  on derived categories of coherent sheaves''}. {\em Duke Math. J.} {\bfseries
  108} no.~1, (2001) 37--108.
  \url{https://doi.org/10.1215/S0012-7094-01-10812-0}.

\end{thebibliography}

\end{document}